\documentclass[12pt]{article}
 
 \usepackage{amsmath,amssymb,amscd,amsthm,esint}
 
\usepackage{graphics,amsmath,amssymb,amsthm,mathrsfs}

\usepackage{graphics,amsmath,amssymb,amsthm,mathrsfs,amsfonts}

\usepackage{sidecap}
\usepackage{float}
\usepackage{extarrows}
\usepackage{booktabs}
\usepackage{verbatim}
\usepackage{hyperref}
\usepackage[usenames,dvipsnames]{xcolor}

\newtheorem{thm}{Theorem}[section]
\newtheorem{lemma}[thm]{Lemma}
\newtheorem{cor}[thm]{Corollary}
\newtheorem{prop}[thm]{Proposition}

\theoremstyle{definition}
\newtheorem{remark}[thm]{Remark}

\def\Xint#1{\mathchoice
      {\XXint\displaystyle\textstyle{#1}}%
      {\XXint\textstyle\scriptstyle{#1}}%
      {\XXint\scriptstyle\scriptscriptstyle{#1}}%
      {\XXint\scriptscriptstyle\scriptscriptstyle{#1}}%
                   \!\int}
\def\XXint#1#2#3{{\setbox0=\hbox{$#1{#2#3}{\int}$}
         \vcenter{\hbox{$#2#3$}}\kern-.5\wd0}}

\def\dashint{\Xint-}
\def\R{\mathbb{R}}
\def\e{\varepsilon}

\numberwithin{equation}{section}

\begin{document}

\title{Asymptotic Expansions of Fundamental Solutions \\  in Parabolic Homogenization}

\author{Jun Geng\thanks{Supported in part  by the NNSF of China (no. 11571152) and Fundamental Research Funds for the Central
Universities (lzujbky-2017-161).} \qquad
Zhongwei Shen\thanks{Supported in part by NSF grant DMS-1600520.}}
\date{}
\maketitle

\begin{abstract}

For a family of second-order parabolic systems with rapidly oscillating and time-dependent periodic coefficients,
 we investigate the asymptotic behavior of fundamental solutions and establish sharp estimates for the remainders.

\end{abstract}

\section{\bf Introduction}\label{section-1}

In this paper we study the asymptotic behavior of fundamental solutions
$\Gamma_\e (x, t; y, s)$ for a family of second-order parabolic operators $\partial_t+\mathcal{L}_\varepsilon$ with
rapidly oscillating and time-dependent periodic coefficients.
Specifically, we consider
\begin{equation}\label{elliptic operator}
\mathcal{L}_\varepsilon=-\text{div}\left(A\big({x}/{\varepsilon},{t}/{\varepsilon^2}\big)\nabla
\right)
\end{equation}
in $\R^d \times \R$, where $\e>0$ and $A(y,s)= \big(a_{ij}^{\alpha\beta} (y,s)\big)$
with $1\le i, j\le d$ and $1\le \alpha, \beta \le m$.
Throughout the paper we will assume that the coefficient matrix
$A=A(y,s)$ is real, bounded measurable  and satisfies the ellipticity condition,
\begin{equation}\label{ellipticity}
\| A\|_\infty  \le \mu^{-1} \quad \text{ and }\quad
\mu |\xi|^2  \le  a^{\alpha\beta}_{ij}(y,s)\xi_i^\alpha
\xi^\beta_j
\end{equation}
 for any 
$\xi=(\xi_i^\alpha ) \in \mathbb{R}^{m\times d} \text{ and  a.e. }  (y,s)\in \mathbb{R}^{d+1}$,
where $\mu >0$.
We also assume that $A$ is 1-periodic; i.e.,  
\begin{equation}\label{periodicity}
A(y+z,s+t)=A(y,s)~~~\text{ for }(z,t)\in \mathbb{Z}^{d+1}\text{ and a.e. }(y,s)\in \mathbb{R}^{d+1}.
\end{equation}
Under these assumptions it is known that as $\e\to 0$, the operator 
 $\partial_t +\mathcal{L}_\e$ 
G-converges to a parabolic operator $\partial_t +\mathcal{L}_0$ with 
constant coefficients \cite{BLP-2011}.

In the scalar case $m=1$,
it follows from a celebrated  theorem of John Nash \cite{Nash} that local solutions of
$(\partial_t +\mathcal{L}_\e)u_\e=0$ are H\"older continuous.
More precisely, if $(\partial_t +\mathcal{L}_\e)u_\e=0$ in $Q_{2r}=Q_{2r}(x_0, t_0)$ for some
$(x_0, t_0)\in \R^{d+1}$ and $0<r<\infty$,
where 
\begin{equation}\label{Q-r}
Q_r(x_0, t_0)=B(x_0, r)\times (t_0-r^2, t_0),
\end{equation}
then there exists some $\sigma\in (0,1)$, depending only on $d$ and $\mu$, such that
\begin{equation}\label{H-0}
\| u_\e\|_{C^{\sigma, \sigma/2}(Q_r)}
\le C r^{-\sigma} \left(\frac{1}{|Q_{2r}|} \int_{Q_{2r}} |u_\e|^2 \right)^{1/2},
\end{equation}
where  $C>0$ depends only on $d$ and $\mu$.
In particular, $C$ and $\sigma$ are independent of $\e>0$.
The periodicity assumption (\ref{periodicity}) is not needed here.
It follows that the fundamental solution $\Gamma_\e (x, t; y,s)$ for $\partial_t+\mathcal{L}_\e$
exists and satisfies the Gaussian estimate
\begin{equation}\label{f-0-1}
| \Gamma_\e (x, t; y, s)|\le 
 \frac{C}{ (t-s)^{\frac{d}{2}}} \exp\left\{ -\frac{\kappa |x-y|^2}{t-s}\right\}
\end{equation}
for any $x, y\in \R^d$ and $-\infty<s<t<\infty$, where $\kappa>0$ depends only on $\mu$ and
$C>0$ depends on $d$ and $\mu$ (also see \cite{Aronson, Fabes-1986} for lower bounds).

If $m\ge 2$, the global H\"older estimate (\ref{H-0}) for $1<r<\infty$ was established
recently in \cite{Geng-Shen-2015} for any $\sigma \in (0,1)$
under the assumptions that
$A$ is elliptic, periodic, and $A\in \text{VMO}_x$
(see (\ref{VMO-x-0}) for the definition of $\text{VMO}_x$).
We mention that the local H\"older estimate for $0<r<\e$ without the periodicity condition 
was obtained earlier in \cite{Byun-2007, Krylov-2007}.
Consequently,  by \cite{Hofmann-2004, Dong-2008}, 
the matrix of fundamental solutions
$\Gamma_\e (x, t; y, s)=\big(\Gamma_\e^{\alpha\beta} (x, t; y, s)\big)$,
with $1\le \alpha, \beta\le m$, exists and satisfies the estimate (\ref{f-0-1}),
where $\kappa>0$ depends only on $\mu$.
The constant $C>0$ in (\ref{f-0-1}) depends on $d$, $m$, $\mu$ and
the function $A^\# (r)$ in (\ref{VMO-x}), but not on $\e>0$.

The primary purpose of this paper is to study the asymptotic behavior, as $\e\to 0$,  of
$\Gamma_\e (x, t; y, s)$, 
$\nabla_x\Gamma_\e (x, t; y,s)$,
$\nabla_y \Gamma_\e (x, t; y, s)$, 
and $\nabla_x\nabla_y \Gamma_\e (x, t; y, s)$.
Our main results extend the analogous estimates for elliptic operators
$-\text{\rm div}\big(A(x/\e)\nabla \big)$ in \cite{AL-1991, KLS-2014} to the parabolic setting.
As demonstrated in the elliptic case \cite{KS-2011}, 
the estimates in this paper open the doors for the use of layer potentials in solving initial-boundary value problems
for the parabolic operator $\partial_t +\mathcal{L}_\e$ with 
sharp estimates that are uniform in $\e>0$.

Let $\Gamma_0 (x,t; y, s)$ denote the matrix of fundamental solutions for the
homogenized operator $\partial_t +\mathcal{L}_0$.
Our first result provides the sharp estimate for 
$\Gamma_\e -\Gamma_0$.

\begin{thm}\label{main-theorem-1}
Suppose that the coefficient matrix $A$ satisfies conditions (\ref{ellipticity}) and (\ref{periodicity}).
If $m\ge 2$, we also assume that $A\in \text{\rm VMO}_x$.
Then
\begin{equation}\label{f-0-3}
|\Gamma_\e (x, t; y, s)-\Gamma_0 (x, t; y, s)|
\le \frac{C\e }{(t-s)^{\frac{d+1}{2}}}
\exp\left\{-\frac{\kappa |x-y|^2}{ t-s}\right\}
\end{equation}
for any $x, y\in \R^d$ and $-\infty<s<t<\infty$,
where $\kappa>0$ depends only on $\mu$.
The constant $C$ depends on $d$, $m$, $\mu$, and  $A^\#$ (if $m\ge 2$). 
 \end{thm}
 
 Let $\chi(y, s)=(\chi_j^{\alpha\beta} (y, s))$, where $1\le j \le d$ and $1\le \alpha, \beta\le m$,
 denote the matrix of correctors for $\partial_t +\mathcal{L}_\e$ (see Section 2 for its definition).
 The next theorem gives an asymptotic expansion for $\nabla_x \Gamma_\e(x, t; y, s)$.
 
 \begin{thm}\label{main-theorem-2}
 Suppose that the coefficient matrix $A$ satisfies conditions (\ref{ellipticity}) and (\ref{periodicity}).
Also assume that $A$ is H\"older continuous, 
\begin{equation}\label{H}
|A(x,t)-A(y, s)|\le \tau \big(|x-y|+|t-s|^{1/2}\big)^\lambda
\end{equation}
for any $(x, t), (y, s)\in \R^{d+1}$, where $\tau\ge 0$ and $\lambda \in (0, 1)$.
Then
\begin{equation}\label{f-0-4}
\aligned
& \big| \nabla_x \Gamma_\e (x, t; y, s)
-\big(1+\nabla \chi(x/\e, t/\e^2)\big)\nabla_x \Gamma_0 (x, t; y, s) \big|\\
&\qquad \le \frac{C\e}{(t-s)^{\frac{d+2}{2}}}
\log (2+\e^{-1} |t-s|^{1/2})
\exp\left\{-\frac{\kappa |x-y|^2}{ t-s}\right\}
\endaligned
\end{equation}
for any $x, y\in \R^d$ and $-\infty<s<t<\infty$,
where $\kappa>0$ depends only on $\mu$.
The constant $C$ depends  on $d$, $m$, $\mu$, and $(\lambda, \tau)$ in (\ref{H}).
 \end{thm}
 
With the summation convention this means that for $1\le i\le d$ and $1\le \alpha, \beta\le m$,
\begin{equation}\label{f-0-5}
\Big|\frac{\partial \Gamma_\e^{\alpha\beta}}{\partial x_i}
(x, t; y, s) 
-\frac{\partial \Gamma_0^{\alpha\beta}} {\partial x_i}
(x, t; y, s) 
-\frac{\partial \chi_j^{\alpha\gamma}}{\partial x_i} (x/\e, t/\e^2)
\frac{\partial \Gamma_0^{\gamma \beta}}{\partial x_j} (x, t; y, s) \Big|
\end{equation}
is bounded by the RHS of (\ref{f-0-4}).
Let $\widetilde{A}(y, s) =\big(\widetilde{a}_{ij}^{\alpha\beta} (y, s)\big)$,
where $\widetilde{a}_{ij}^{\alpha\beta} (y, s)=a_{ji}^{\beta\alpha}(y, -s)$.
Let $\widetilde{\Gamma}_\e (x, t; y, s)=\big(\widetilde{\Gamma}_\e^{\alpha\beta} (x, t; y, s)\big)$ denote 
the matrix of fundamental solutions for the operator $\partial_t +\widetilde{\mathcal{L}}_\e$,
where $\widetilde{\mathcal{L}}_\e= -\text{\rm div}\big(\widetilde{A} (x/\e, t/\e^2)\nabla \big)$.
Then
\begin{equation}\label{f-0-6}
\Gamma_\e^{\alpha\beta} (x, t; y, s)=\widetilde{\Gamma}_\e^{\beta\alpha} (y, -s; x, -t).
\end{equation}
Since $\widetilde{A}$ satisfies the same conditions as $A$, 
it follows from (\ref{f-0-4}), (\ref{f-0-5}) and (\ref{f-0-6}) that
\begin{equation}\label{f-0-7}
\Big|\frac{\partial \Gamma_\e^{\beta\alpha}}{\partial y_i}
(x, t; y, s) 
-\frac{\partial \Gamma_0^{\beta\alpha}} {\partial y_i}
(x, t; y, s) 
-\frac{\partial \widetilde{\chi}_j^{\alpha\gamma}}{\partial y_i} (y/\e, -s/\e^2)
\frac{\partial \Gamma_0^{ \beta\gamma}}{\partial y_j} (x, t; y, s) \Big|
\end{equation}
is bounded by the RHS of (\ref{f-0-4}),
where $\widetilde{\chi} (y, s)=\big(\widetilde{\chi}_j^{\alpha\beta} (y, s)\big)$ denotes the correctors for
 $\partial_t +\widetilde{\mathcal{L}}_\e$. That is
\begin{equation}\label{f-0-8}
\aligned
& \big| \nabla_y \Gamma^T_\e (x, t; y, s)
-\big(I+\nabla \widetilde{\chi}(y/\e, -s/\e^2)\big)\nabla_y \Gamma^T_0 (x, t; y, s) \big|\\
&\qquad \le \frac{C\e}{(t-s)^{\frac{d+2}{2}}}
\log (2+\e^{-1} |t-s|^{1/2})
\exp\left\{-\frac{\kappa |x-y|^2}{ t-s}\right\},
\endaligned
\end{equation}
where $\Gamma_\e^T$ denotes the transpose of the matrix $\Gamma_\e$.

We also obtain an asymptotic expansion for $\nabla_x\nabla_y \Gamma_\e (x, t; y, s)$.

\begin{thm}\label{main-theorem-3}
Under the same assumptions on $A$ as in Theorem \ref{main-theorem-2},
the estimate
\begin{equation}\label{f-0-9}
\aligned
&\Big|\frac{\partial}{\partial x_i \partial y_j} \Big\{ \Gamma^{\alpha\beta}_\e (x, t; y, s)\Big\}- \\
& \frac{\partial }{\partial x_i}
\Big\{ \delta^{\alpha\gamma} x_k + \e \chi_k^{\alpha\gamma} (x/\e, t/\e^2) \Big\}
\frac{\partial^2}{\partial x_k \partial y_\ell}
\Big\{ \Gamma_0^{\gamma \sigma} (x, t; y, s) \Big\}
\frac{\partial}{\partial y_j}
\Big\{ \delta^{\beta\sigma} y_\ell +\e \widetilde{\chi}_\ell^{\beta\sigma} (y/\e, -s/\e^2)\Big\} \Big|\\
& \qquad \le \frac{C\e}{(t-s)^{\frac{d+3}{2}}}
\log (2+\e^{-1}  |t-s|^{1/2})
\exp\left\{-\frac{\kappa |x-y|^2}{ t-s}\right\}
\endaligned
\end{equation}
holds for $x, y\in \R^d$ and $-\infty<s<t<\infty$, where $\kappa$ depends only on $\mu$.
The constant $C$ depends  on $d$, $m$, $\mu$, and $(\lambda, \tau)$ in (\ref{H}).
\end{thm}

In the scale case $m=1$, the estimate (\ref{f-0-3}),
{\it without} the exponential factor, is known under the conditions that
$A$ is elliptic, periodic, symmetric, and time-independent (see \cite[p.77]{JKO-1994} and its references).
This was proved by using the Floquet-Bloch decomposition of the fundamental solutions 
and by studying the spectral properties of elliptic operators \newline
$
-(\nabla +ik )\cdot A (\nabla +ik)
$
 in a periodic cell, where $i=\sqrt{-1}$ and  $k\in \R^d$.
Such an approach is not available when the coefficient matrix $A$ is time-dependent.
To the best of authors' knowledge,  the Gaussian bound in Theorem \ref{main-theorem-1} as well
our estimates in Theorems \ref{main-theorem-2} and \ref{main-theorem-3}
is new even in the case that $m=1$ and $A$ is time-independent.

As a corollary of Theorems \ref{main-theorem-1} and \ref{main-theorem-2},
 we establish an interesting result on equistabilization for time-dependent coefficients (cf. \cite[p.77]{JKO-1994}).

\begin{cor}\label{cor-0-1}
Assume that $A$ satisfies the same conditions as in Theorem \ref{main-theorem-1}.
Let $f\in L^\infty(\R^d)$ and $u_\e$ be the bounded solution of the Cauchy problem,
\begin{equation}\label{C-P}
\left\{
\aligned
(\partial_t +\mathcal{L}_\e)u_\e & = 0& \quad  &\text{ in } \R^d \times (0, \infty),\\
u_\e & =f & \quad  & \text{ on } \R^d \times \{ t=0\},
\endaligned
\right.
\end{equation}
with $\e=1$ or $0$.  Then for any $x\in \R^d$ and $t\ge 1$, 
\begin{equation}\label{C-P-1}
|u_1 (x, t) -u_0(x, t)| \le C t^{-1/2} \| f\|_\infty.
\end{equation}
Furthermore, if $A$ is H\"older continuous, 
\begin{equation}\label{C-P-2}
\Big|\nabla u^\alpha_1 (x, t)-\nabla u^\alpha_0(x, t) 
-\nabla \chi_j^{\alpha\beta} (x, t) \frac{\partial u_0^\beta}{\partial x_j} (x, t) \Big|\le
C t^{-1} \log (2+ t) \| f\|_\infty
\end{equation}
for any $x\in \R^d$ and $t\ge  1$.
\end{cor}

We now describe some of the key ideas in the proof of Theorems \ref{main-theorem-1}, \ref{main-theorem-2} and
\ref{main-theorem-3}.
As indicated earlier, our main results
extend the analogous results in \cite{AL-1991, KLS-2014} for the elliptic operators $-\text{\rm div} \big(A(x/\e)\nabla \big)$, where
$A=A(y)$ is elliptic and periodic.
Our general approach  is inspired by the work  in \cite{KLS-2014}, which uses a two-scale expansion and 
relies on regularity estimates that are uniform in
$\e>0$.
 Following  \cite{Geng-Shen-2017},  we consider the two-scale expansion
 \begin{equation}\label{w-1}
 w_\e =u_\e (x, t)-u_0 (x, t)
 -\e \chi (x/\e, t/\e^2)  S_\e (\nabla u_0) 
 -\e^2  \phi (x/\e, t/\e^2) \nabla S_\e (\nabla u_0),
 \end{equation}
 where $\chi (y, s)$ and $\phi(y,s)$ are correctors and dual correctors respectively for
 $\partial_t +\mathcal{L}_\e$ (see Section 2 for their definitions).
 In (\ref{w-1}) the operator $S_\e$  is a parabolic smoothing
 operator at scale $\e$.
 In comparison with the elliptic case, an extra term is added in the RHS of (\ref{w-1}).
 This modification allows us to show if $(\partial_t +\mathcal{L}_\e)u_\e=(\partial_t+\mathcal{L}_0) u_0$, then
 \begin{equation}\label{w-2}
  (\partial_t +\mathcal{L}_\e) w_\e=\e\,  \text{\rm div}(F_\e)
  \end{equation}
  for some function $F_\e$, which depends only on $u_0$.
  As a consequence, we may apply the uniform interior $L^\infty$ estimates established in \cite{Geng-Shen-2015} to the function $w_\e$.
  To fully utilize the ideas above, we will consider the functions
  \begin{equation}\label{w-3}
  \aligned
  u_\e (x, t) & =\int_{-\infty}^t \int_{\R^d} \Gamma_\e (x, t; y, s) f(y, s) e^{-\psi (y)}\, dy ds,\\
   u_0 (x, t) & =\int_{-\infty}^t \int_{\R^d} \Gamma_0 (x, t; y, s) f(y, s) e^{-\psi (y)}\, dy ds,
   \endaligned
   \end{equation}
  where $\psi$ is a Lipchitz function in $\R^d$ and $f\in C_0^\infty(Q_r(y_0, s_0); \R^m)$.
  The main technical step in proving Theorem \ref{main-theorem-1} involves bounding the $L^\infty$ norm
  \begin{equation}\label{w-4}
  \|e^\psi (u_\e -u_0)\|_{L^\infty( Q_r (x_0, t_0))}
  \end{equation}
  by $\|f\|_{L^2(Q_r (y_0, s_0))}$, where $0<\e< r=c\sqrt{t_0-s_0}$.
  We remark that the use of weighted inequalities with weight $e^\psi$ to generate the exponential factor
  in the Gaussian bound is more or less well known. 
  Our approach may be regarded as a variation of the standard one found in \cite{Hofmann-2004, Dong-2008}
  (also see earlier work in \cite{Fabes-1986, Davies-1987, Davies-1987a}).
  
  The proof of Theorem \ref{main-theorem-2} uses the estimate in Theorem \ref{main-theorem-1}.
  The stronger assumption that $A$ is H\"older continuos allows us to 
  apply the uniform interior Lipschitz estimate obtained in \cite{Geng-Shen-2015} to the function $w_\e$ in (\ref{w-1}).
  To see Theorem \ref{main-theorem-3}, one uses the fact that as a function of $(x, t)$,
  $\nabla_y \Gamma_\e (x, t; y, s)$ is a solution of 
  $(\partial_t +\mathcal{L}_\e ) u_\e=0$, away from the pole $(y, s)$.  
  
We end this section with some notations that will be used throughout the paper.
A function $h=h(y,s)$ in $\R^{d+1}$
 is said to be $1$-periodic if $h$ is periodic with respect to $\mathbb{Z}^{d+1}$. 
 We will use the notations
 $$
\fint_E f =\frac{1}{|E|} \int_E f \quad \text{ and } \quad  h^\e (x,t)= h (x/\e, t/\e^2)
 $$
 for $\e>0$, as well as the summation convention that the repeated indices are summed.
 Finally, we shall use $\kappa$ to denote positive constants that depend only on $\mu$, 
 and $C$ constants that depend at most on $d$, $m$, $\mu$
 and the smoothness of $A$, but never on $\e$.



\section{\bf Preliminaries}\label{section-2}

 Let $\mathcal{L}_\varepsilon=-\text{div}\left(A^\e(x,t) \nabla\right)$, where $A^\e (x,t)=A(x/\e, t/\e^2)$.
 Assume that $A(y, s)$ is 1-periodic in $(y, s)$ and satisfies   the ellipticity condition (\ref{ellipticity}).
 For $1\leq j\leq d$ and $1\le \beta\le m$, 
 the corrector $\chi_j^\beta=\chi_j^\beta (y,s)=(\chi_{j}^{\alpha\beta} (y, s))$ 
  is defined as the weak solution of the following cell problem:
\begin{equation}\label{corrector}
\begin{cases}
\big(\partial_s +\mathcal{L}_1\big) (\chi_j^\beta)
=-\mathcal{L}_1(P_j^\beta ) ~~~\text{in}~~Y, \\
\chi_j^\beta =\chi^\beta_j(y,s)~~ \text{is } \text{1-periodic in } (y,s),\\
\int_{Y} \chi_j^\beta = 0,
\end{cases}
\end{equation}
where $Y=[0,1)^{d+1}$, $P_j^\beta  (y)=y_j e^\beta$, and
$e^\beta=(0, \dots, 1, \dots, 0 )$ with $1$ in the $\beta^{th}$ position. 
Note that 
\begin{equation}\label{c-00}
(\partial_s+\mathcal{L}_1)(\chi_j^\beta +P_j^\beta )=0~~~\text{in}~~\mathbb{R}^{d+1}.
\end{equation}
By the rescaling property of $\partial_t +\mathcal{L}_\e$, 
one obtains that
\begin{equation}
(\partial_t+\mathcal{L}_\varepsilon)\left\{\varepsilon\chi_j^\beta(x/\varepsilon,t/\varepsilon^2)+P^\beta_j(x)\right\}
=0~~~\text{in}~~\mathbb{R}^{d+1}.
\end{equation}

 We say $A\in \text{VMO}_x$ if
\begin{equation}\label{VMO-x-0}
\lim_{ r\to 0} A^{\#}(r)=0,
\end{equation}
 where
\begin{equation}\label{VMO-x}
A^{\#} (r)
=\sup_{\substack{0< \rho<r\\ (x, t)\in \R^{d+1}}}
\fint_{t-\rho^2}^{t}
\fint_{y\in B(x, \rho)}\fint_{z\in B(x, \rho)}
|A(y, s)-A(z, s)|\, dz dy ds.
\end{equation}
Observe that if $A(y, s)$ is continuous in the variable $y$, uniformly in $(y, s)$, then $A\in \text{VMO}_x$.

\begin{lemma}\label{p-1}
Assume that $A(y, s)$ is 1-periodic in $(y, s)$ and satisfies  (\ref{ellipticity}).
If $m\ge 2$, we also assume $A\in \text{\rm VMO}_x$.
Then $\chi_j^\beta\in L^\infty (Y; \R^m)$.
\end{lemma}

\begin{proof}
In the scalar case $m=1$, this follows from (\ref{c-00}) by  Nash's classical estimate.
Moreover, the estimate 
\begin{equation}\label{p-2}
\left(\fint_{Q_r (x, t)} |\nabla \chi_j^\beta |^2\right)^{1/2}
\le C r^{\sigma-1}
\end{equation}
holds for any $0<r<1$ and $(x, t)\in \R^{d+1}$,
where  $Q_r (x,t)=B(x, r)\times (t-r^2, t)$, 
and $C>0$ and $\sigma\in (0,1)$ depend on $d$ and $\mu$.
If $m\ge 2$ and $A\in \text{VMO}_x$,
the boundedness of $\chi_j^\beta$  follows from the interior $W^{1, p}$ estimates for local solutions of
$(\partial_t +\mathcal{L}_1)(u)=\text{div}(f)$ \cite{Byun-2007, Krylov-2007}.
In this case the estimate (\ref{p-2}) holds for any $\sigma \in (0,1)$.
\end{proof}

Let $\widehat{A}=(\widehat{a}^{\alpha\beta}_{ij})$, where $1\leq i,j\leq d$, $1\le \alpha, \beta\le m$,
and
\begin{align}\label{A}
\widehat{a}^{\alpha\beta}_{ij}=\dashint_{Y}\left[a^{\alpha\beta}_{ij}+a^{\alpha\gamma}
_{i k}\frac{\partial}{\partial y_k}\chi^{\gamma\beta}_{j}\right];
\end{align}
that is
$$
\widehat{A}=\dashint_Y \Big\{ A +A\nabla \chi \Big\}.
$$
It is known that the constant matrix  $\widehat{A}$ satisfies the ellipticity condition,
\begin{equation}\label{A-ellipticity}
\mu |\xi|^2 \le \widehat{a}_{ij}^{\alpha\beta} \xi_i^\alpha \xi_j^\beta
 \qquad \text{ for any } \xi=(\xi_j^\beta) \in \R^{m\times d},
\end{equation}
and $|\widehat{a}_{ij}^{\alpha\beta}|\le \mu_1$,
where $\mu_1>0$ depends only on $d$, $m$ and $\mu$ \cite{BLP-2011}.
Denote 
$$
\mathcal{L}_0=-\text{div}(\widehat{A}\nabla).
$$
 Then $\partial_t+\mathcal{L}_0$ is the homogenized operator 
for the family of parabolic operators  $\partial_t+\mathcal{L}_\varepsilon$, $\e>0$.

To introduce the dual correctors, we consider the 1-periodic matrix-valued function
\begin{equation}\label{B}
B= A + A\nabla \chi -\widehat{A}.
\end{equation}
More precisely,  $B=B(y, s)= \big( b_{ij}^{\alpha\beta}\big)$, where
$1\le i, j\le d$, $1\le \alpha, \beta\le m$, and
\begin{equation}\label{b}
b_{ij}^{\alpha\beta}=a_{ij}^{\alpha\beta}+a_{ik}^{\alpha\gamma}
\frac{\partial \chi^{\gamma\beta}_j}{\partial y_k}-\widehat{a}^{\alpha\beta}_{ij}.
\end{equation}

\begin{lemma}\label{1.15}
Let $1\leq j\leq d$ and $1\le \alpha, \beta\le m$.
Then there exist 1-periodic functions 
$\phi_{kij}^{\alpha\beta}(y,s)$ in $\R^{d+1}$ such that 
$\phi_{kij}^{\alpha\beta}\in H^1(Y)$,
\begin{equation}\label{1.10}
b_{ij}^{\alpha\beta}=\frac{\partial}{\partial y_k}(\phi^{\alpha\beta}_{kij})~~\text{ and }~~\phi^{\alpha\beta}_{kij}
=-\phi_{ikj}^{\alpha\beta},
\end{equation}
where $1\le k, i\le d+1$, $b_{ij}^{\alpha\beta}$ is defined by
(\ref{b}) for $1\le i\le d$,  $b_{(d+1)j}^{\alpha\beta}=-\chi_j^{\alpha\beta}$,
and we have used the notation $y_{d+1}=s$.
\end{lemma}

\begin{proof}
This lemma was proved in \cite{Geng-Shen-2015}.
We give a proof here for reader's convenience.
By (\ref{corrector}) and (\ref{A}), $b_{ij}^{\alpha\beta} \in L^2 (Y)$ and 
\begin{equation}\label{1.8-1}
\int_{Y} b_{ij}^{\alpha \beta}=0
\end{equation}
for $1\le i\le d+1$.
It follows that  there exist $f_{ij}^{\alpha\beta} \in H^2 (Y)$ such that
\begin{equation}\label{1.8-2}
\begin{cases}
\Delta_{d+1} f_{ij}^{\alpha\beta}=b^{\alpha\beta}_{ij}~~~~\text{ in } \R^{d+1}, \\
f_{ij}^{\alpha\beta} ~~\text{is  1-periodic}~~~\text{ in } \R^{d+1}, 
\end{cases}
\end{equation}
where $\Delta_{d+1}$ denotes  the Laplacian in $\R^{d+1}$. Write 
\begin{align}\label{1.11}
b_{ij}^{\alpha\beta}=
\frac{\partial}{\partial y_k}\left\{\frac{\partial}{\partial y_k}f_{ij}^{\alpha\beta}
-\frac{\partial}{\partial y_i}f_{kj}^{\alpha\beta}\right\}+\frac{\partial}{\partial y_i}\left\{\frac{\partial}{\partial y_k}f_{kj}^{\alpha\beta}\right\},
\end{align}
where the index $k$ is summed from $1$ to $d+1$.
Note that by (\ref{corrector}),
 \begin{equation}\label{div}
 \sum_{i=1}^{d+1} \frac{\partial b^{\alpha\beta}_{ij}}{\partial y_i}
 =\sum_{i=1}^d \frac{\partial}{\partial y_i}b^{\alpha\beta}_{ij}-\frac{\partial}{\partial s}\chi_j^{\alpha\beta}=0.
 \end{equation}
In view of (\ref{1.8-2}) this implies that 
$$
\sum_{i=1}^{d+1}\frac{\partial }{\partial y_i}  f_{ij}^{\alpha\beta}
$$
is harmonic in $\R^{d+1}$. Since it is 1-periodic, it must be constant.
Consequently, by (\ref{1.11}), we obtain 
\begin{align}\label{1.13}
b_{ij}^{\alpha\beta}=\frac{\partial}{\partial y_k}(\phi_{kij}^{\alpha\beta}),
\end{align}
where
\begin{align}\label{1.14}
\phi_{kij}^{\alpha\beta}
=\frac{\partial}{\partial y_k}f_{ij}^{\alpha\beta}
-\frac{\partial}{\partial y_i}f_{kj}^{\alpha\beta}
\end{align}
is 1-periodic and belongs to $H^1(Y)$.
It is easy to see that $\phi_{kij}^{\alpha\beta}=-\phi_{ikj}^{\alpha\beta}$. 
\end{proof}

The 1-periodic functions $(\phi_{kij}^{\alpha\beta})$, given by Lemma \ref{1.15},
 are called dual correctors for
the family of parabolic operators $\partial_t +\mathcal{L}_\e$, $\e>0$.

\begin{lemma}\label{p-10}
Let $\phi= (\phi_{kij}^{\alpha\beta})$ be the dual correctors, given 
by Lemma \ref{1.15}.
Under the same assumption as in Lemma \ref{p-1},
one has $ \phi_{kij}^{\alpha\beta}\in L^\infty (Y)$.
\end{lemma}

\begin{proof}
It follows from (\ref{p-2}) that if $(x, t)\in \R^{d+1}$ and $0<r<1$,
\begin{equation}\label{p-11}
\int_{Q_r (x, t)} |b_{ij}^{\alpha\beta}|^2 
\le C r^{d+2\sigma}
\end{equation}
for some $\sigma \in (0,1)$.
By covering the interval $(t-r, t)$ with intervals of length $r^2$, we
obtain 
$$
\int_{B_r (x, t)} |b_{ij}^{\alpha\beta}|^2 \le C r^{d-1 +2\sigma},
$$
where $B_r (x, t)=B(x, r)\times (t-r, t)$. 
Hence, by H\"older's inequality,
$$
\int_{B_r (x, t)} |b_{ij}^{\alpha\beta}|
\le C r^{d+\sigma}.
$$
Thus, for any $(x, t)\in Y$,
\begin{equation}\label{p-12}
\aligned
\int_Y \frac{|b_{ij}^{\alpha\beta} (y, s)|}{(|x-y| +|t-s|)^d}\, dyds
& \le
C \sum_{j=1}^\infty 2^{jd}
\int_{|y-x| +|t-s|\sim 2^{-j}}
 | b_{ij}^{\alpha\beta} (y, s)|\, dyds\\
&\le C.
\endaligned
\end{equation}
In view of (\ref{1.8-2}), by using the fundamental solution for $\Delta_{d+1}$ in $\R^{d+1}$,
we may show that
$$
\|\nabla_{y, s} f_{ij}^{\alpha\beta}\|_{L^\infty (Y)}
\le C \| \nabla_{y, s} f_{ij}^{\alpha\beta}\|_{L^2(Y)}
+\sup_{(x, t)\in Y} \int_{Y}
\frac{|b_{ij}^{\alpha\beta} (y, s)|}{ (|x-y| +|t-s|)^d}\, dyds,
$$
where $\nabla_{y, s}$ denotes the gradient in $\R^{d+1}$.
This, together with (\ref{p-12}), shows that $|\nabla_{y, s} f_{ij}^{\alpha\beta}|\in L^\infty (Y)$.
By (\ref{1.14}) we obtain $\phi_{kij}^{\alpha\beta}\in L^\infty(Y)$.
\end{proof}

\begin{remark}\label{re-2.0}
{\rm
Suppose $A=A(y, s)$ is H\"older continuous in $(y,s)$.
By (\ref{c-00}) and the standard  regularity theory for $\partial_s +\mathcal{L}_1$,
$\nabla \chi(y, s)$ is H\"older continuous in $(y, s)$.
It follows that $b_{ij}^{\alpha\beta} (y, s)$ is H\"older continuous in $(y, s)$.
In view of (\ref{1.8-2}) and (\ref{1.14}) one may deduce that
$\nabla_{y, s} \phi_{kij}^{\alpha\beta}$ is H\"older continuous in $(y, s)$.
This will be used in the proof of Theorems \ref{main-theorem-2} and \ref{main-theorem-3}.
}
\end{remark}

\begin{thm}\label{L-theorem}
Suppose that $A$ satisfies the conditions (\ref{ellipticity}) and (\ref{periodicity}).
If $m\ge 2$, we also assume $A\in \text{VMO}_x$.
Let $u_\e$ be a weak solution of $(\partial_t +\mathcal{L}_\e) u_\e =\text{\rm div} (f)$
in $Q_{2r}=Q_{2r}(x_0, t_0)$ for some $0<r<\infty$, where
$f=(f_i^\alpha)\in L^p(Q_{2r}; \R^{m\times d})$ for some $p>d+2$.
Then
\begin{equation}\label{L-estimate}
\|u_\e \|_{L^\infty(Q_r)}
\le C \left\{ \left(\fint_{Q_{2r}} |u_\e|^2\right)^{1/2}
+  r \left(\fint_{Q_{2r}} |f|^p\right)^{1/p} \right\},
\end{equation}
where $C$ depends only on $d$, $m$, $p$, $\mu$, and $A^\#$ in (\ref{VMO-x}) (if $m\ge 2$).
\end{thm}

\begin{proof}
If $m=1$, this follows from the well known Nash's estimate.
The periodicity is not needed.
If $m\ge 2$, (\ref{L-estimate}) follows from the uniform interior H\"older estimate
in \cite[Theorem 1.1]{Geng-Shen-2015}.
\end{proof}

Under the assumptions on $A$ in  Theorem \ref{L-theorem},
 the matrix of fundamental solutions for 
 $\partial_t +\mathcal{L}_\e$ in $\R^{d+1}$ exists and satisfies the Gaussian estimate (\ref{f-0-1}).
 This follows from the $L^\infty$ estimate (\ref{L-estimate}) by a general result in \cite{Hofmann-2004}
 (also see \cite{Auscher, Dong-2008}).
 
 \begin{thm}\label{Lip-theorem}
 Suppose that $A$ satisfies conditions (\ref{ellipticity}) and (\ref{periodicity}).
 Also assume that $A$ satisfies the H\"older condition (\ref{H}).
 Let $u_\e$ be a weak solution of $(\partial_t +\mathcal{L}_\e)u_\e=F$
 in $Q_{2r}=Q_{2r} (x_0, t_0)$ for some $0<r<\infty$, where $F\in L^p(Q_{2r}; \R^m)$ for some
 $p>d+2$. Then
 \begin{equation}\label{Lip-estimate}
 \|\nabla u_\e\|_{L^\infty(Q_r)}
 \le C\left\{ \frac{1}{r} \left(\fint_{Q_{2r}} |u_\e|^2\right)^{1/2}
+  r \left(\fint_{Q_{2r}} |F|^p\right)^{1/p} \right\},
\end{equation}
where $C$ depends only on $d$, $m$, $p$, $\mu$, and $(\lambda, \tau) $ in (\ref{H}).
\end{thm}
 
 \begin{proof}
 This was proved in \cite[Theorem 1.2]{Geng-Shen-2015}.
 \end{proof}
 
 The Lipschitz estimate (\ref{Lip-estimate}) allows us to bound
 $\nabla_x \Gamma_\e (x, t; y, s)$,
 $\nabla_y \Gamma_\e (x, t; y, s)$
 and $\nabla_x\nabla_y \Gamma_\e (x, t; y, s)$.
 
 \begin{thm}\label{D-bound}
 Assume that $A$ satisfies the same conditions as in Theorem \ref{Lip-theorem}.
 Then
 \begin{equation}\label{D-1}
 | \nabla_x \Gamma_\e (x, t; y, s)|
 +| \nabla_y \Gamma_\e (x, t; y, s)|
 \le 
 \frac{C}{ (t-s)^{\frac{d+1}{2}}} \exp\left\{ -\frac{\kappa |x-y|^2}{t-s}\right\},
 \end{equation}
 \begin{equation}\label{D-2}
 | \nabla_x \nabla_y \Gamma_\e (x, t; y, s)|
 \le 
 \frac{C}{ (t-s)^{\frac{d+2}{2}}} \exp\left\{ -\frac{\kappa |x-y|^2}{t-s}\right\},
 \end{equation}
 for any $x, y\in \R^d$ and $-\infty<s<t<\infty$,
 where $\kappa>0$ depends only on $\mu$.
 The constant $C$ depends  on $d$, $m$, $\mu$, and $(\lambda, \tau) $ in (\ref{H}).
\end{thm} 

\begin{proof}
Fix $x_0, y_0\in \R^d$ and $s_0<t_0$.
Let $u_\e (x, t)=\Gamma_\e (x, t; y_0, s_0)$.
Then $(\partial_t +\mathcal{L}_\e) u_\e =0$ in $Q_{2r} (x_0, t_0)$,
where $r=\sqrt{t_0-s_0}/8$.
The estimate for $|\nabla_x \Gamma (x_0, t_0; y_0, s_0)|$
now follows from (\ref{Lip-estimate}) and (\ref{f-0-1}) (with a different $\kappa$).
In view of (\ref{f-0-6}) this also gives the estimate for $|\nabla_y \Gamma_\e (x_0, t_0; y_0, s_0)|$.
Finally, to see (\ref{D-2}),
we let $v_\e (x, t)=\nabla_y \Gamma_\e (x, t; y_0, s_0)$.
Then $(\partial_t +\mathcal{L}_\e) v_\e =0$ in $Q_{2r} (x_0, t_0)$.
By applying (\ref{Lip-estimate}) to $v_\e$ and using the estimate
in (\ref{D-1}) for $\nabla_y \Gamma _\e (x, t; y, s)$, we obtain the desired estimate
for $|\nabla_x \nabla_y \Gamma_\e (x_0, t_0; y_0, s_0)|$.
\end{proof}



\section{A two-scale expansion}\label{section-3}

 Suppose that
 \begin{equation}\label{e-1}
 (\partial_t +\mathcal{L}_\e) u_\e
 =(\partial_t +\mathcal{L}_0) u_0
 \end{equation}
 in $\Omega \times (T_0, T_1)$, where $\Omega\subset \R^d$.
 Let $S_\e$ be a linear operator to be chosen later.
 Following \cite{Geng-Shen-2017},
 we consider the two-scale expansion $w_\e =(w_\e^\alpha)$, where
\begin{equation}\label{w}
\aligned
w_\e ^\alpha (x, t)  = u_\e^\alpha (x, t) -u_0^\alpha  (x, t) &  -\e \chi_j^{\alpha\beta} (x/\e, t/\e^2)
S_\e \left(\frac{\partial u_0^\beta}{\partial x_j}\right)\\
&-\e^2 \phi_{(d+1) ij}^{\alpha\beta} (x/\e, t/\e^2)
\frac{\partial}{\partial x_i} S_\e \left(\frac{\partial u_0^\beta}{\partial x_j}\right),
\endaligned
\end{equation}
and $\chi_j^{\alpha\beta}$, $\phi_{(d+1)ij}^{\alpha\beta}$
are the correctors and dual correctors introduced in the last section. 
The repeated indices $i, j$ in (\ref{w}) are summed from $1$ to $d$.

\begin{prop}\label{e-2}
Let $u_\e\in L^2(T_0, T_1; H^1(\Omega))$ and
$u_0 \in L^2(T_0, T_1; H^2(\Omega))$.
Let  $w_\e$ be defined by (\ref{w}).
 Assume (\ref{e-1}) holds in $\Omega \times (T_0, T_1)$.
Then
\begin{equation}\label{e-3}
(\partial_ t +\mathcal{L}_\e) w_\e =\e\, \text{\rm div}(F_\e)
\end{equation}
in $\Omega \times (T_0, T_1)$, where $F_\e = (F_{\e, i}^\alpha)$ and
\begin{equation}\label{e-4}
\aligned
F_{\e, i}^\alpha (x, t)
 = &  \e^{-1} \left( a_{ij}^{\alpha\beta} (x/\e, t/\e^2)-\widehat{a}_{ij}^{\alpha\beta} \right)
\left( \frac{\partial u_0^\beta}{\partial x_j} -S_\e \left(\frac{\partial u_0^\beta}{\partial x_j}\right)\right)\\
& + a_{ij}^{\alpha\beta} (x/\e, t/\e^2) \chi_k^{\beta\gamma} (x/\e, t/\e^2) 
\frac{\partial}{\partial x_j} S_\e \left(\frac{\partial u_0^\gamma}{\partial x_k}\right)\\
& +  \phi_{ikj}^{\alpha\beta} (x/\e, t/\e^2)
\frac{\partial}{\partial x_k} S_\e\left(\frac{\partial u_0^\beta}{\partial x_j}\right)\\
&+\e \phi_{i(d+1)j}^{\alpha\beta}(x/\e, t/\e^2) \partial_t S_\e \left(\frac{\partial u_0^\beta}{\partial x_j}\right)\\
& - a_{ij}^{\alpha\beta} (x/\e, t/\e^2) \left(\frac{\partial }{\partial x_j}
(\phi^{\beta\gamma}_{(d+1)\ell k}) \right)(x/\e, t/\e^2)
\frac{\partial}{\partial x_\ell} S_\e \left(\frac{\partial u_0^\gamma}{\partial x_k}\right)\\
& -\e a_{ij}^{\alpha\beta} (x/\e, t/\e^2)
\phi_{(d+1) \ell k}^{\beta\gamma} (x/\e, t/\e^2)
\frac{\partial^2}{\partial x_j \partial x_\ell}
S_\e \left(\frac{\partial u_0^\gamma}{\partial x_k}\right).
\endaligned
\end{equation}
The repeated indices  $i, j, k, \ell$ are summed from  $1$ to $d$.
\end{prop}

\begin{proof}
This proposition was proved in \cite[Theorem 2.2]{Geng-Shen-2017}.
\end{proof}

We now introduce a parabolic  smoothing operator.
Let 
$$
\mathcal{O}=\{ (x,t)\in \R^{d+1}: |x|^2 +|t|\le 1 \}.
$$
Fix a nonnegative function  $\theta=\theta (x,t)
\in C_0^\infty(\mathcal{O})$ such that 
$\int_{\mathbb{R}^{d+1}}\theta=1$. 
Let $\theta_\e (x, t)=\e^{-d-2} \theta (x/\e, t/\e^2)$.
Define
\begin{equation}\label{1.18}
S_\varepsilon(f)(x,t)=f* \theta_\e (x, t)=\int_{\mathbb{R}^{d+1}}f(x-y,t-s)
\theta_\e (y,s)\, dyds.
\end{equation}

\begin{lemma}\label{lemma-S-3}
Let $g=g(x,t)$ be a 1-periodic function in $(x,t)$ and
 $\psi=\psi (x)$ a bounded Lipschitz function in $\R^d$.
 Then
\begin{align}\label{1.22}
\| e^\psi g^\e S_\varepsilon (f)\|_{L^p({\mathbb{R}^{d+1}})}\leq C\, e^{\e \|\nabla \psi\|_\infty}
\| g \|_{L^p(Y)}  \| e^\psi f \|_{L^p({\mathbb{R}^{d+1}})}
\end{align}
for any $1\le p<\infty$,
where  $g^\e (x, t)= g(x/\e, t/\e^2)$ and $C$ depends only on $d$ and $p$.
\end{lemma}

\begin{proof}
 Using $\int_{\R^{d+1}} \theta_\e =1$ and H\"older's inequality, we obtain 
$$
|S_\e (e^{-\psi} f) (x, t)|^p
\le \int_{\R^{d+1}} |e^{-\psi (y)} f(y, s)|^p \, \theta_\e (x-y, t-s)\, dyds.
$$
It follows that 
$$
\aligned
|e^{\psi (x)} S_\e (e^{-\psi} f) (x, t)|^p
 & \le \int_{\R^{d+1}} |e^{\psi (x)-\psi (y)} f(y, s)|^p \, \theta_\e (x-y, t-s)\, dyds\\
& \le e^{ \e  p\|\nabla \psi\|_\infty} 
\int_{\R^{d+1}} | f(y, s)|^p \, \theta_\e (x-y, t-s)\, dyds,
\endaligned
$$
where we have used the facts that
$|\psi (x)-\psi(y)|\le \|\nabla \psi\|_\infty |x-y|$
and $\theta_\e (x-y, t-s)=0$ if $|x-y|>\e$, for the last step. Hence,
by Fubini's Theorem, 
$$
\aligned
 &\int_{\R^{d+1}} |g^\e(x,t)|^p | e^\psi S_\e (e^{-\psi} f)(x, t)|^p\, dx dt\\
& \le  e^{\e p \|\nabla \psi\|_\infty}
\sup_{(y, s)\in \R^{d+1}}
\int_{\R^{d+1}} | g^\e(x, t)|^p \theta_\e (x-y, t-s) \, dx dt \int_{\R^{d+1}} | f(y, s)|^p\, dyds\\
&\le C\,  e^{\e p\|\nabla \psi\|_\infty}
\|  g \|^p_{L^p(Y)} 
\| f \|_{L^p(\R^{d+1})}^p,
\endaligned
$$
where $C$ depends only on $d$.
This gives (\ref{1.22}).
\end{proof}

\begin{remark}\label{remark-S}
{\rm
Let $\Omega\subset \R^d$ and $(T_0, T_1)\subset \R$.
Define
\begin{equation}\label{O-e}
\Omega_\e =\big\{ x\in \R^d: \, \text{\rm dist}(x, \Omega)<\e\big\}.
\end{equation}
Observe that for $(x, t)\in \Omega\times (T_0, T_1)$,
$S_\e (f)(x, t)=S_\e (f\eta_\e)(x, t)$, 
where $\eta_\e =\eta_\e (x, t)$ is the 
characteristic function of $\Omega_\e \times (T_0-\e^2, T_1+\e^2)$.
By applying (\ref{1.22}) to the function $f\eta_\e$, 
one may deduce that
\begin{equation}\label{1.22-0}
\int_{T_0}^{T_1}\!\! \int_{\Omega} |e^{\psi} g^\e  S_\e (f) |^p \, dxdt
\le C e^{\e p \|\nabla \psi\|_\infty} 
\|g\|_{L^p(Y)}^p  \int_{T_0-\e^2}^{T_1+\e^2}\!\!\! \int_{\Omega_\e} |e^\psi f|^p\, dxdt.
\end{equation}
Using $\int_{\R^{d+1}} |\nabla \theta_\e |\, dxdt\le C\e^{-1}$,
the same argument as in the proof of Lemma \ref{lemma-S-3} also shows that
\begin{equation}\label{S-remark-1}
\int_{T_0}^{T_1}\!\!
 \int_{\Omega} |e^{\psi} g^\e  \nabla S_\e (f) |^p \, dxdt
\le C \e^{-p}e^{\e p \|\nabla \psi\|_\infty} 
\|g\|_{L^p(Y)}^p  \int_{T_0-\e^2}^{T_1+\e^2} \!\!\!\int_{\Omega_\e} |e^\psi f|^p\, dxdt
 \end{equation}
for $1\le p<\infty$, where $C$ depends only on $d$ and $p$.
}
\end{remark}

\begin{lemma}\label{lemma-S-2}
Let $S_\e$ be defined as in (\ref{1.18}). 
Let $1\le p<\infty$ and $\psi$ be a bounded Lipschitz function in $\R^d$. Then for 
$\Omega\subset \R^d$ and $(T_0, T_1)\subset \R$,
\begin{equation}\label{S-approx}
\aligned 
& \int_{T_0}^{T_1}\!\!\int_{\Omega}
| e^\psi \big(  S_\e (\nabla f) -\nabla f\big)|^p\, dxdt\\
& \quad
\le C \e^p  e^{\e p\| \nabla \psi\|_\infty}
 \int_{T_0-\e^2}^{T_1+\e^2}\!\!\!\int_{\Omega_\e}
|e^\psi \big(  |\nabla^2 f |
+ |\partial_t f | \big)|^p \, dx dt,
\endaligned
\end{equation}
where $\Omega_\e$ is given by (\ref{O-e}) and 
$C$ depends only on $d$ and $p$.
\end{lemma}

\begin{proof} Write
$$
S_\e (\nabla f) (x, t) -\nabla f(x, t)=J_1 (x, t) +J_2 (x, t),
$$
where
$$
\aligned
J_1(x, t)&=\int_{\R^{d+1}} \theta_\e (y, s) \big( \nabla f (x-y, t-s) -\nabla f (x-y, t)\big)\, dyds,\\
J_2(x, t) &=
\int_{\R^{d+1}} \theta_\e (y, s) \big( \nabla f( x-y, t) -\nabla f (x, t)\big)\, dyds.
\endaligned
$$
To estimate $J_2$, we observe that by H\"older's inequality and the fact $\int_{\R^{d+1}} \theta_\e \, dyds=1$,
$$
|J_2(x, t)|^p 
\le \int_{\R^{d+1}} \theta_\e (y, s) | \nabla f (x-y, t) -\nabla f (x, t)|^p\, dyds,
$$
and
$$
\aligned
|\nabla f (x-y, t) -\nabla f (x, t)|
&=\Big|\int_0^1 \frac{\partial }{\partial \tau} 
\nabla f( x-\tau y, t) \, d\tau \Big|\\
& \le | y|\int_0^1 |\nabla^2 f(x-\tau y, t)|\, d\tau \\
&\le |y| \left(\int_0^1 |\nabla^2 f(x-\tau y, t)|^p\, d\tau\right)^{1/p}.
\endaligned
$$
It follows by Fubini's Theorem that
$$
\aligned
&\int_{T_0}^{T_1}\!\! \int_{\Omega} | e^{\psi (x)} J_2 (x, t)|^p\, dx dt\\
&\le \int_{T_0}^{T_1}\!\!\int_{\Omega}\int_{\R^{d+1}}\int_0^1
e^{p \psi (x)} \theta_\e (y, s) |y|^p |\nabla^2 f (x-\tau y, t)|^p\, d\tau dyds dxdt\\
&\le \e^p e^{ \e p \|\nabla \psi\|_\infty}
\int_{T_0}^{T_1}\!\! \int_{\Omega}\int_{\R^{d+1}}\int_0^1
e^{p \psi (x-\tau y)} \theta_\e (y, s)  |\nabla^2 f (x-\tau y, t)|^p\, d\tau dyds dxdt\\
&\le  \e^p e^{\e p\|\nabla \psi\|_\infty}
\int_{T_0}^{T_1} \!\!\int_{\Omega_\e} | e^\psi \nabla^2 f|^p\, dx dt,
\endaligned
$$
where we have used the facts that
$|\psi (x)-\psi (x-\tau y)|\le |\tau| |y|\|\nabla \psi\|_\infty$ and
$\theta_\e (y, s)=0$ if $|y|>\e$.

Finally, to estimate $J_1$, we first use integration by parts to obtain
$$
|J_1(x, t)|\le \int_{\R^{d+1}}
|\nabla \theta_\e (y, s)| | f(x-y, t-s)- f(x-y, t)|\, dyds.
$$
By H\"older's inequality,
$$
|J_1(x, t)|^p \le  C \e^{1-p}\int_{\R^{d+1}}
|\nabla \theta_\e (y, s)| | f(x-y, t-s)- f(x-y, t)|^p\, dyds,
$$
where we have also used the fact $\int_{\R^{d+1}} |\nabla \theta_\e| \, dyds \le C \e^{-1}$.
Using
$$
\aligned
|f (x-y, t-s)- f(x-y, t)|
&\le \Big| \int_0^1 \frac{\partial}{\partial \tau} f(x-y, t-\tau s)\, d\tau \Big|\\
&\le |s| \left(\int_0^1 |\partial_t f (x-y, t-\tau s)|^p \, d\tau \right)^{1/p},
\endaligned
$$
we see that by Fubini's Theorem, 
$$
\aligned
& \int_{T_0}^{T_1}\!\!\int_{\Omega} | e^{\psi (x)} J_1 (x, t)|^p\, dx dt\\
&\le  C\e^{1-p} \int_{T_0}^{T_1}\!\int_{\Omega}\int_{\R^{d+1}}\int_0^1
e^{p \psi (x)}|\nabla  \theta_\e (y, s)| |s|^p |\partial_t f (x- y, t-\tau s)|^p\, d\tau dyds dxdt\\
&\le C \e^{1+p} e^{ \e p \|\nabla \psi\|_\infty}
\int_{T_0}^{T_1}\!\int_{\Omega}\int_{\R^{d+1}}\int_0^1
e^{p \psi (x- y)} |\nabla \theta_\e (y, s)|  |\partial_t f (x- y, t-\tau s)|^p\, d\tau dyds dxdt\\
&\le  C \e^p e^{\e p\|\nabla \psi\|_\infty}
\int_{T_0-\e^2}^{T_1+\e^2} \!\!\! \int_{\Omega_\e} |e^\psi \partial_t f|^p\, dxdt,
\endaligned
$$
where we have used the facts that 
$|\psi (x)-\psi (x-y)|\le \|\nabla \psi\|_\infty |y|$ and
$\theta_\e (y,s)=0$ if $|y|>\e$ or $|s|>\e^2$.
This, together with the estimate for $J_2$, completes the proof.
\end{proof}

\begin{thm}\label{e-10}
Let $F_\e =(F_{\e, i}^\alpha)$ be given by (\ref{e-4}) and $1\le p<\infty$.
Then for any $\Omega \subset \R^d$ and $(T_0, T_1)\subset \R$,
\begin{equation}\label{e-11}
\int_{T_0}^{T_1} \!\!\int_\Omega
|e^\psi F_\e  |^p\, dx dt
\le C e^{\e p \|\nabla \psi\|_\infty}
\int_{T_0-\e^2}^{T_1+\e^2} \!\!\int_{\Omega_\e}
\Big\{
|e^\psi \nabla^2 u_0|^p +
| e^\psi \partial_t u_0|^p \Big\} dx dt,
\end{equation}
 where 
$\Omega_\e$ is given by (\ref{O-e}) and
$C$ depends only on $d$, $m$, $p$ and $\mu$.
\end{thm}

\begin{proof}
Observe that 
\begin{equation}\label{e-12}
\aligned
 \int_{T_0}^{T_1} \!\!\int_\Omega
|e^\psi F_\e  |^p\, dx dt 
& \le  C \e^{-p}  \int_{T_0}^{T_1}\! \int_{\Omega}
|\nabla u_0 -S_\e (\nabla u_0)|^p e^{p \psi} \, dxdt\\
& \quad+ C  \int_{T_0}^{T_1} \!\int_{\Omega} |\chi^\e |^p | S_\e (\nabla^2 u_0)|^p e^{p\psi}\, dxdt\\
&\quad+C   \int_{T_0}^{T_1}\! \int_{\Omega} |\phi^\e|^p | S_\e (\nabla^2 u_0)|^p e^{p\psi}\, dxdt\\
&\quad + C \e^{p} \int_{T_0}^{T_1} \!\int_{\Omega}
|\phi^\e |^p |\nabla S_\e (\partial_t u_0)|^p e^{p\psi} dxdt\\
&\quad+C  \int_{T_0}^{T_1} \!\int_{\Omega} |(\nabla \phi)^\e|^p
| S_\e (\nabla^2 u_0)|^p e^{p\psi} dx dt\\
&\quad 
 + C \e^{p} \int_{T_0}^{T_1} \!\int_{\Omega} |\phi^\e |^p |\nabla S_\e (\nabla^2 u_0)|^p e^{p\psi} dxdt,
\endaligned
\end{equation}
where $C$ depends only on $d$ and $\mu$.
In (\ref{e-12}) we have also used the observation that
$\partial_t S_\e (\nabla u_0)=\nabla S_\e (\partial_t u_0)$
and $\nabla S_\e (\nabla u_0)=S_\e (\nabla^2 u_0)$.

We now proceed to bound each term in the RHS of (\ref{e-12}), using  Lemma \ref{lemma-S-2} and
Remark \ref{remark-S}.
By Lemma \ref{lemma-S-2}, the first term in the RHS of (\ref{e-12}) is bounded by
\begin{equation}\label{e-13}
C  e^{ p\e \|\nabla \psi\|_\infty}
\int_{T_0-\e^2}^{T_1+\e^2}\!\!\! \int_{\Omega_\e}
|e^\psi (|\nabla^2 u_0| +|\partial_t u_0|)|^p\, dx dt.
\end{equation}
Using  (\ref{1.22-0})  we may bound the second, third, fifth terms in the RHS of (\ref{e-12})  by
\begin{equation}\label{e-14}
C  e^{p\e \|\nabla \psi\|_\infty}
\int_{T_0-e^2}^{T_1+\e^2}\!\!\! \int_{\Omega_\e}
|e^\psi \nabla^2 u_0|^p\, dx dt.
\end{equation}
Finally, by (\ref{S-remark-1}), the fourth and sixth terms in the RHS of (\ref{e-12}) are bounded by
(\ref{e-13}).
This completes the proof.
\end{proof}



\section{Weighted estimates for $\partial_t +\mathcal{L}_0$}\label{section-4}

Recall that  $\Gamma_0 (x, t; y, s)$ denotes the matrix of fundamental solutions 
for the homogenized operator $\partial_t +\mathcal{L}_0$ in $\R^{d+1}$.
Since $\mathcal{L}_0$ is a second-order elliptic operator with constant coefficients,
it is well known that $\Gamma_0 (x, t; y, s)=\Gamma_0 (x-y, t-s; 0, 0)$ and 
\begin{equation}\label{f-const}
|\nabla_x^M\partial_t^N \Gamma_0 (x, t; y, s)|
\le 
 \frac{C}{ (t-s)^{\frac{d+M+2N}{2}}} \exp\left\{ -\frac{\kappa |x-y|^2}{t-s}\right\}
\end{equation}
for any $M, N\ge 0$,
where $\kappa>0$ depends only on $\mu$, and
$C$ depends on $d$, $m$, $M$, $N$ and $\mu$.

Let $\psi: \R^d\to \R$ be a bounded Lipschitz function and
\begin{equation}\label{w-0-1}
u_0(x, t)=\int_{-\infty}^t \int_{\R^d} \Gamma_0 (x, t; y, s) f(y, s) e^{-\psi (y)}\, dyds, 
\end{equation}
where $f\in C_0^\infty (\R^{d+1}; \R^m)$.
Then
\begin{equation}\label{w-0-0}
\left(\partial_t +\mathcal{L}_0\right) u_0 =e^{-\psi} f \quad \text{ in } \R^{d+1}.
\end{equation}
The goal of this section is to prove the following.

\begin{thm}\label{w-0-2}
Let $u_0$ be defined by (\ref{w-0-1}). 
Suppose that $f(x, t)=0$ for $t\le s_0$.
Then
\begin{equation}\label{w-0-3}
\int_{s_0}^{t}\! \int_{\R^d}
\big|e^{\psi} \big( |\nabla^2 u_0| +|\partial_t u_0|\big)\big|^2\, dx dt
\le C  e^{\kappa (t-s_0) \|\nabla \psi\|_\infty^2}
 \int_{s_0}^{t}\! \int_{\R^d} | f|^2\, dxdt
\end{equation}
for any $s_0<t<\infty$,
where $\kappa>0$ depends only on $\mu$ and
$C$  depend only on $d$ and $\mu$.
\end{thm}

We start with an estimate on a lower order term.

\begin{lemma}\label{w-00}
Let $u_0$ be defined by (\ref{w-0-1}). 
Suppose that $f(x, t)=0$ for $t<s_0$.
Then
\begin{equation}\label{w-00-1}
\int_{s_0}^{t}\!\int_{\R^d}
 | e^{\psi} \nabla u_0|^2 \, dxdt
 \le C (t-s_0) e^{\kappa_1 (t-s_0) \|\nabla \psi\|_\infty^2}
 \int_{s_0}^{t}\!\int_{\R^d} | f|^2\, dxdt
 \end{equation}
 for any $s_0<t<\infty$,
where $\kappa_1>0$ depends only on $\mu$ and
$C$  depend only on $d$ and  $\mu$.
\end{lemma}

\begin{proof}
It follows from (\ref{f-const}) that for $x, y\in \R^d$ and $t>s$,
$$
\aligned
e^{\psi (x)-\psi (y)} |\nabla _x \Gamma_0 (x, t; y, s)|
 &\le \frac{C}{(t-s)^{\frac{d+1}{2}}}
\exp \left\{ \psi (x)-\psi(y) -\frac{\kappa |x-y|^2}{t-s}\right\}\\
&  \le \frac{C}{(t-s)^{\frac{d+1}{2}}}
\exp \left\{ \|\nabla \psi\|_\infty |x-y|-\frac{\kappa |x-y|^2}{t-s}\right\}.
\endaligned
$$
This, together with the inequality,
\begin{equation}\label{w-00-2}
|\nabla \psi\|_\infty |x-y|
\le 
\frac{ (t-s)\|\nabla \psi\|^2_\infty}{2\kappa}
+\frac{\kappa |x-y|^2}{2 (t-s)},
\end{equation}
yields 
\begin{equation}\label{w-00-3}
e^{\psi (x)-\psi (y)} |\nabla _x \Gamma_0 (x, t; y, s)|
\le C e^{\frac{ (t-s)\|\nabla \psi\|^2_\infty}{2\kappa}}
\cdot \frac{1}{(t-s)^{\frac{d+1}{2}}}
e^{-\frac{\kappa |x-y|^2}{2 (t-s)}}.
\end{equation}
It follows that
$$
\aligned
|e^{\psi (x)} \nabla u_0 (x, t)|
 &\le \int_{s_0}^t\! \int_{\R^d}
 e^{\psi (x)-\psi (y)} |\nabla_x \Gamma_0 (x, t; y, s)| | f(y, s)|\, dyds\\
 &\le C  e^{\frac{ (t-s_0)\|\nabla \psi\|^2_\infty}{2\kappa}}
 \int_{s_0}^t\! \int_{\R^d} 
\frac{1}{(t-s)^{\frac{d+1}{2}}}
e^{-\frac{\kappa |x-y|^2}{2 (t-s)}} | f(y, s)|\, dy ds\\
&\le 
C  e^{\frac{ (t-s_0)\|\nabla \psi\|^2_\infty}{2\kappa}}
(t-s_0)^{\frac{1}{4}}
 \left(\int_{s_0}^t \!\int_{\R^d} 
\frac{1}{(t-s)^{\frac{d+1}{2}}}
e^{-\frac{\kappa |x-y|^2}{2 (t-s)}} | f(y, s)|^2\, dy ds\right)^{1/2},
\endaligned
$$
where we have used H\"older's inequality for the last step.
The estimate (\ref{w-00-1}) now follows by
Fubini's Theorem.
\end{proof}

\begin{proof}[\bf Proof of Theorem \ref{w-0-2}]

In view of (\ref{w-0-0}) we have
$$
(\partial_t +\mathcal{L}_0) \frac{\partial u_0}{\partial x_k} =\frac{\partial}{\partial x_k} (e^{-\psi} f )
$$
in $\mathbb{R}^{d+1}$.
It follows that
$$
\int_{\R^d} \partial_t \nabla u_0 \cdot (\nabla u_0) e^{2\psi}\, dx
-\int_{\R^d} \widehat{a}_{ij}^{\alpha\beta} \frac{\partial^3 u^\beta_0}{\partial x_i \partial x_j \partial x_k} \cdot \frac{\partial u_0^\alpha}
{\partial x_k}  e^{2\psi}\, dx
=\int_{\R^d} \frac{\partial}{\partial x_k} (e^{-\psi} f^\alpha )
\frac{\partial u_0^\alpha}{\partial x_k} e^{2\psi}\, dx.
$$
Using integration  by parts, we obtain 
$$
\aligned
&\frac12 \frac{d}{dt}\int_{\R^d} |\nabla u_0|^2 e^{2\psi}\, dx
+\int_{\R^d} \widehat{a}_{ij}^{\alpha\beta} \frac{\partial^2 u^\beta_0}{ \partial x_j \partial x_k} \cdot \frac{\partial^2 u_0^\alpha}
{\partial x_i\partial x_k}  e^{2\psi}\, dx\\
&=-\int_{\R^d}  f \cdot (\Delta u_0) e^{\psi}\, dx
-\int_{\R^d} e^{-\psi}  f^\alpha \frac{\partial u_0^\alpha}{\partial x_k} \frac{\partial e^{2\psi}}{\partial x_k}\, dx
-\int_{\R^d} \widehat{a}_{ij}^{\alpha\beta} \frac{\partial^2 u^\beta_0}{\partial x_j \partial x_k} \cdot \frac{\partial u_0^\alpha}
{\partial x_k}  \frac{\partial e^{2\psi}}{\partial x_i} \, dx.
\endaligned
$$
By the ellipticity of $\mathcal{L}_0$,
this yields
$$
\aligned
& \frac12 \frac{d}{dt}\int_{\R^d} |\nabla u_0|^2 e^{2\psi}\, dx
+\mu \int_{\R^d} |\nabla^2 u_0|^2 e^{2\psi}\, dx\\
&
\quad \le C \int_{\R^d} |f| |\nabla^2 u_0| e^\psi \, dx
+ C \int_{\R^d} |f|^2\, dx
+ C \|\nabla \psi\|_\infty^2 \int_{\R^d} |\nabla u_0|^2 e^{2\psi}\, dx\\
& \qquad\qquad\qquad
+ C \|\nabla \psi\|_\infty \int_{\R^d} |\nabla^2 u_0| |\nabla u_0| e^{2\psi}\, dx,
\endaligned
$$
where $C$ depends only on $d$ and $\mu$.
Using Cauchy inequality, we may further deduce that 
$$
\aligned
& \frac12 \frac{d}{dt}\int_{\R^d} |\nabla u_0|^2 e^{2\psi}\, dx
+\frac{\mu}{2} \int_{\R^d} |\nabla^2 u_0|^2 e^{2\psi}\, dx\\
&
\quad \le
 C \int_{\R^d} |f|^2\, dx
+ C \|\nabla \psi\|_\infty^2 \int_{\R^d} |\nabla u_0|^2 e^{2\psi}\, dx.\\
\endaligned
$$
We now integrate the inequality above in $t$ over the interval $(s_0, s_1)$.
This leads to 
\begin{equation}\label{w-00-10}
\aligned
& \frac12 \int_{\R^d} |\nabla u_0 (x, s_1)|^2 e^{2\psi}\, dx
+\frac{\mu}{2} \int_{s_0}^{s_1} \!\!\int_{\R^d} |\nabla^2 u_0|^2 e^{2\psi}\, dxdt \\
&
\quad \le
 C \int_{s_0}^{s_1}\!\!\int_{\R^d} |f|^2\, dxdt
+ C \|\nabla \psi\|_\infty^2\int_{s_0}^{s_1} \!\!\int_{\R^d} |\nabla u_0|^2 e^{2\psi}\, dxdt\\
&\quad \le C e^{\kappa (s_1-s_0)\|\nabla \psi\|_\infty^2}
\int_{s_0}^{s_1} \!\!\int_{\R^d} |f|^2\, dx dt,
\endaligned
\end{equation}
where we have used (\ref{w-00-1}) for the last inequality.
Estimate (\ref{w-0-3}) follows readily from (\ref{w-00-10}).
\end{proof}



\section{Proof of Theorem \ref{main-theorem-1}}\label{section-5}

We start with some weighted estimates.

\begin{lemma}\label{w-1-1}
Suppose that
\begin{equation}\label{w-1-2}
\left\{
\aligned
(\partial_t +\mathcal{L}_\e) w_\e  & = \e\, \text{\rm div} (F_\e) & \quad & \text{ in } \R^d \times (s_0, \infty),\\
w_\e  & =0 & \quad & \text{ on } \R^d \times \{t=s_0\}.
\endaligned
\right.
\end{equation}
Let $\psi: \R^d \to \R$ be a bounded Lipschitz function.
Then for any $t>s_0$,
\begin{equation}\label{w-1-3}
\int_{\R^d} |w_\e (x, t)|^2 e^{2\psi (x)}\, dx
\le C\e^2  e^{\kappa (t-s_0)\|\nabla \psi\|^2_\infty}
\int_{s_0}^{t}\!\int_{\R^d} |F_\e (x, s)|^2 e^{2\psi (x)} \, dx ds,
\end{equation}
where $\kappa>0$  and $C>0$ depends only on  $\mu$.
\end{lemma}

\begin{proof}
Let
\begin{equation}\label{w-1-4}
I(t)=\int_{\R^d} | w_\e (x, t)|^2 e^{2\psi (x)} \, dx.
\end{equation}
Note that
$$
\aligned
I^\prime (t)
&=2\int_{\R^d} \langle \partial_t w_\e, e^{2\psi} w_\e \rangle \, dx\\
&=-2 \int_{\R^d} \langle \mathcal{L}_\e (w_\e), e^{2\psi} w_\e \rangle\, dx
+2\e \int_{\R^d} \langle \text{\rm div} (F_\e), e^{2\psi} w_\e \rangle\, dx\\
&=-2\int_{\R^d}
A^\e  \nabla w_\e \cdot \nabla \left(e^{2\psi} w_\e \right)\, dx
-2 \e \int_{\R^d} F_\e \cdot \nabla \left(e^{2\psi} w_\e \right)\, dx\\
&=-2\int_{\R^d}
A^\e \nabla w_\e \cdot (\nabla  w_\e) e^{2\psi}\, dx
-2\int_{\R^d}
A^\e \nabla w_\e \cdot \nabla (e^{2\psi}) w_\e \, dx\\
& \qquad
-2\e \int_{\R^d} F_\e \cdot (\nabla  w_\e) e^{2\psi}\, dx
-2 \e\int_{\R^d} F_\e \cdot \nabla (e^{2\psi}) w_\e\, dx,
\endaligned
$$
where $\langle\, ,\,  \rangle$ denotes the pairing in $H^{-1} (\R^d; \R^m) \times H^1(\R^d; \R^m)$.
It follows that
$$
\aligned
I^\prime (t)
 & \le -2\mu \int_{\R^d} |\nabla w_\e|^2 e^{2\psi}\, dx
+ \kappa \|\nabla \psi\|_\infty \int_{\R^d} |\nabla w_\e | |w_\e| e^{2\psi}\, dx\\
&\qquad
+2 \e \int_{\R^d} |\nabla w_\e| |F_\e| e^{2\psi}\, dx
+ 4\e  \|\nabla \psi\|_\infty \int_{\R^d} |w_\e | |F_\e| e^{2\psi}\, dx,
\endaligned
$$
where $\kappa>0 $ depends only on $\mu$.
By Cauchy inequality this implies that
\begin{equation}\label{w-1-5}
I^\prime (t)
\le \kappa \|\nabla \psi\|_\infty^2 I (t) +\kappa \e^2   \int_{\R^d} |F_\e(x, t)|^2 e^{2\psi}\, dx,
\end{equation}
where $\kappa>0 $ depends only on $\mu$.
Hence,
$$
\frac{d}{dt}
\left\{ I(t) e^{-\kappa  (t-s_0) \|\nabla \psi\|_\infty^2} \right\}
\le C \e^2 e^{-\kappa  (t-s_0)\|\nabla \psi\|^2_\infty}
\int_{\R^d} |F_\e (x, t)|^2 e^{2\psi}\, dx.
$$
Since $I(s_0)=0$,
it follows that
$$
\aligned
I(t) &\le 
C \e^2\int_{s_0}^t\! \int_{\R^d}
e^{\kappa (t-s)\|\nabla \psi\|_\infty^2}
 |F_\e (x, s)|^2 e^{2\psi} \, dx ds\\
 &\le  C\e^2 e^{\kappa (t-s_0) \|\nabla \psi\|_\infty^2}
 \int_{s_0}^t\! \int_{\R^d} |F_\e (x, s)|^2 e^{2\psi}\, dx ds.
 \endaligned
$$
This completes the proof.
\end{proof}

\begin{lemma}\label{w-1-6}
Suppose that $u_\e \in L^2((-\infty, T); H^1(\R^d))$ and $u_0\in L^2((-\infty, T);  H^2(\R^d))$ for any $T\in \R$,
 that 
$$
\left\{
\aligned
& (\partial_t +\mathcal{L}_\e) u_\e= (\partial_t +\mathcal{L}_0) u_0 & \quad & \text{ in } \R^{d+1},\\
& u_\e (x, t)=u_0 (x, t)=0 &\quad & \text{ for } t\le s_0.
\endaligned
\right.
$$
Let $w_\e$ be defined by (\ref{w}), where the operator $S_\e$ is given by (\ref{1.18}).
Then for any $t>s_0$,
\begin{equation}\label{w-1-7}
\aligned
& \int_{\R^d} |w_\e (x, t)|^2 e^{2\psi (x)}\, dx\\
& \le C \e^2 e^{2 \e \|\nabla \psi\|_\infty + \kappa (t-s_0)\|\nabla \psi\|^2_\infty}
\int_{s_0}^{t+\e^2}\!\!\!\! \int_{\R^d}
\Big\{ |\nabla^2 u_0 (x, s)|^2 +|\partial_s u_0 (x, s)|^2 \Big\} e^{2\psi(x)}  dx ds,
\endaligned
\end{equation}
where $\psi$ is a bounded Lipschitz function in $\R^d$,
$\kappa$ depends only on $\mu$,  and $C$ depends only on $d$, $m$ and $\mu$.
\end{lemma}

\begin{proof}
This follows readily from Lemma \ref{w-1-1} and Theorem \ref{e-10} with $p=2$.
\end{proof}

The next theorem gives a weighted $L^\infty$ estimate.

\begin{thm}\label{w-1-9}
Assume that $A$ is 1-periodic and satisfies (\ref{ellipticity}).
If $m\ge 2$, we also assume that $A\in \text{\rm VMO}_x$.
Suppose that 
$
(\partial_t +\mathcal{L}_\e) u_\e =(\partial_t +\mathcal{L}_0)u_0
$
in $B(x_0,3r)\times (t_0-5r^2, t_0+r^2)$ for some $(x_0, t_0)\in \R^{d+1}$ and
$\e\le r<\infty$.
Then 
\begin{equation}\label{w-1-9.0}
\aligned
& \|e^\psi (u_\e-u_0)\|_{L^\infty(Q_r (x_0, t_0))}\\
& \le C e^{3r\|\nabla \psi\|_\infty}
 \left(\fint_{Q_{2r} (x_0, t_0)} |e^\psi (u_\e -u_0)|^2\right)^{1/2}\\
& \quad+C \e r e^{3r \|\nabla \psi\|_\infty}
\| e^\psi (|\nabla^2 u_0|  +|\partial_t u_0|) \|_{L^\infty( B(x_0, 3r)\times (t_0-5r^2, t_0 +r^2))}\\
& \quad+ C \e  e^{3r\|\nabla \psi\|_\infty} \| e^\psi \nabla u_0\|_{L^\infty( B(x_0, 3r)\times (t_0-5r^2, t_0 +r^2))},
\endaligned
\end{equation}
where $\psi$ is a Lipschitz function in $\R^d$ and
$C$ depends only on $d$, $m$, $\mu$ and $A^\#$ (if $m\ge 2$).
\end{thm}

\begin{proof}
Let $w_\e$ be defined by (\ref{w}).
Then $(\partial_t +\mathcal{L}_\e)w_\e=\e\, \text{\rm div}(F_\e)$ in $Q_{2r}(x_0, t_0)$,
where $F_\e$ is given by (\ref{e-4}).
It follows by Theorem \ref{L-theorem} that
\begin{equation}\label{w-1-10}
\|w_\e\|_{L^\infty(Q_r(x_0, t_0))}
\le C \left\{ \left(\fint_{Q_{2r} (x_0, t_0)} |w_\e|^2\right)^{1/2}
+\e r \left(\fint_{Q_{2r} (x_0, t_0)} |F_\e|^p\right)^{1/p}\right\},
\end{equation}
where $p>d+2$.
This leads to 
$$
\aligned
\|u_\e -u_0\|_{L^\infty(Q_r(x_0, t_0))}
  \le & C  \left(\fint_{Q_{2r} (x_0, t_0)} | u_\e-u_0|^2\right)^{1/2}
 + C \e r
 \left(\fint_{Q_{2r} (x_0, t_0)} |F_\e|^p\right)^{1/p}\\
& + C \e   \| S_\e (\nabla u_0)\|_{L^\infty(Q_{2r}(x_0, t_0))}
 +C \e^2   \|S_\e (\nabla^2 u_0)\|_{L^\infty (Q_{2r}(x_0, t_0))},
\endaligned
$$
where we have used the boundedness of $\chi$ and $\phi$ in Lemmas \ref{p-1} and \ref{p-10}.
Hence, using $|\psi (x)-\psi (y)|\le 2r\|\nabla \psi\|_\infty$ for $x, y\in B(x_0, 2r)$, we obtain
\begin{equation}\label{w-1-10.1}
\aligned
\|e^{\psi}(u_\e -u_0)\|_{L^\infty(Q_r(x_0, t_0))}
  \le & C e^{2r\|\nabla \psi\|_\infty}  \left(\fint_{Q_{2r} (x_0, t_0)} |e^\psi (u_\e-u_0)|^2\right)^{1/2}\\
& + C\e  re^{2r \|\nabla \psi\|_\infty}
 \left(\fint_{Q_{2r} (x_0, t_0)} |e^\psi F_\e|^p\right)^{1/p}\\
& + C \e  e^{2r\|\nabla \psi\|_\infty} \|e^\psi S_\e (\nabla u_0)\|_{L^\infty(Q_{2r}(x_0, t_0))}\\
& +C \e^2  e^{2r\|\nabla \psi\|_\infty} \|e^\psi S_\e (\nabla^2 u_0)\|_{L^\infty (Q_{2r}(x_0, t_0))}.
\endaligned
\end{equation}
Finally, we use Theorem \ref{e-10} to bound the second term in the
RHS of (\ref{w-1-10.1}).
This yields 
$$
\aligned
& \|e^{\psi}(u_\e -u_0)\|_{L^\infty(Q_r(x_0, t_0))}\\
   &\le  C e^{2r\|\nabla \psi\|_\infty}  \left(\fint_{Q_{2r} (x_0, t_0)} |e^\psi (u_\e-u_0)|^2\right)^{1/2}\\
& \quad + C\e  re^{3r \|\nabla \psi\|_\infty}
 \left(\fint_{t_0-5r^2}^{t_0+r^2}\fint_{B(x_0, 3r)}
 \Big\{ |e^\psi \nabla^2 u_0|^p 
 +|e^\psi \partial_t u_0|^p \Big\} \right)^{1/p}\\
& \quad+ C \e  e^{2r\|\nabla \psi\|_\infty} \|e^\psi S_\e (\nabla u_0)\|_{L^\infty(Q_{2r}(x_0, t_0))}
+C \e^2  e^{2r\|\nabla \psi\|_\infty} \|e^\psi S_\e (\nabla^2 u_0)\|_{L^\infty (Q_{2r}(x_0, t_0))},
\endaligned
$$
where $p>d+2$ and we also used the assumption $\e\le r$.
Estimate (\ref{w-1-9.0}) now follows.
\end{proof}

We are now in a position to give the proof of Theorem \ref{main-theorem-1}.

\begin{proof}[\bf Proof of Theorem \ref{main-theorem-1}]
We begin by fixing $x_0, y_0\in \R^{d+1}$ and $s_0<t_0$.
We may assume that 
$$
\e< r=(t_0-s_0)^{1/2}/100.
$$
 For otherwise the desired estimate (\ref{f-0-3}) follows
directly from (\ref{f-0-1}).

For $f\in C_0^\infty (Q_{r} (y_0, s_0); \R^m)$, define
$$
\aligned
u_\e (x, t) & =\int_{-\infty}^t \int_{\R^d} e^{-\psi(y)}\Gamma_\e (x, t; y, s) f(y, s)\, dyds,\\
u_0 (x, t) &=\int_{-\infty}^t \int_{\R^d} e^{-\psi(y)}\Gamma_0 (x, t; y, s)  f(y, s) \, dyds,
\endaligned
$$
where $\psi$ is a bounded  Lipschitz function in $\R^d$ to be chosen later.
Then
$$
(\partial_t +\mathcal{L}_\e) u_\e = (\partial_t +\mathcal{L}_0) u_0 =e^{-\psi} f \quad \text{ in } \R^{d+1}
$$
and $u_\e (x, t)=u_0 (x, t)=0$ if $t\le s_0$.
Let $w_\e$ be defined by (\ref{w}).
It follows from Lemma \ref{w-1-6} and Theorem \ref{w-0-2} that
\begin{equation}\label{main-t}
\aligned
\int_{\R^d} |w_\e (x, t)|^2 e^{2\psi (x)}\, dx
\le C \e^2 e^{2\e \|\nabla \psi\|_\infty + \kappa (t-s_0+\e^2)\|\nabla \psi\|_\infty^2}
\int_{s_0}^{t+\e^2}\!\! \!\int_{\R^d} |f|^2\, dx ds
\endaligned
\end{equation}
for any $t>s_0$.

Next, we use (\ref{w-1-9.0}) to obtain 
\begin{equation}\label{w-1-15}
\aligned
& \|e^\psi (u_\e-u_0)\|_{L^\infty(Q_r (x_0, t_0))}\\
& \le C e^{3r\|\nabla \psi\|_\infty}
 \left(\fint_{Q_{2r} (x_0, t_0)} |e^\psi w_\e|^2\right)^{1/2}\\
& \quad+C \e r e^{3r \|\nabla \psi\|_\infty}
\| e^\psi (|\nabla^2 u_0|  +|\partial_t u_0|) \|_{L^\infty( B(x_0, 3r)\times (t_0-5r^2, t_0 +r^2))}\\
& \quad+ C \e  e^{3r\|\nabla \psi\|_\infty} \| e^\psi \nabla u_0\|_{L^\infty( B(x_0, 3r)\times (t_0-5r^2, t_0 +r^2))}.
\endaligned
\end{equation}
Since supp$(f)\subset Q_{r} (y_0, s_0)$, it follows from the estimate (\ref{f-const}) 
for $\Gamma_0 (x, t; y, s)$ that
\begin{equation}\label{w-1-16}
\aligned
 |\nabla^2 u_0 (x, t)| &
+|\partial_t u_0 (x, t)| + r^{-1}|\nabla u_0 (x, t)|\\
& \le C \exp\left\{-\frac{\kappa |x_0-y_0|^2}{ t_0-s_0}\right\}
\fint_{Q_{r}(y_0, s_0)} |fe^{-\psi}|\, dyds
\endaligned
\end{equation}
for any $x\in B(x_0, 3r)$ and $|t-t_0|\le 5 r^2$, where $\kappa>0$ depends only on $\mu$.
Thus, by (\ref{w-1-15}), we obtain 
\begin{equation}\label{w-1-22}
\aligned
& \| e^\psi (u_\e -u_0)\|_{L^\infty(Q_{r}(x_0, t_0))}\\
&\le C e^{3r \|\nabla \psi\|_\infty}
 \left(\fint_{Q_{2r}(x_0, t_0)} |e^\psi w_\e|^2\right)^{1/2}\\
&\qquad +  \e r e^{c(|x_0-y_0|+r)\|\nabla \psi\|_\infty}
\exp\left\{-\frac{\kappa |x_0-y_0|^2}{t_0-s_0 }\right\}
\fint_{Q_{r}(y_0, s_0)} |f|\, dyds\\
&\le C \e r e^{c r \|\nabla \psi\|_\infty}
\left\{ e^{cr^2 \|\nabla \psi\|_\infty^2}
+ e^{c|x_0-y_0|\|\nabla \psi\|_\infty}
\exp\left\{-\frac{\kappa |x_0-y_0|^2}{t_0-s_0}\right\}\right\}\\
&\qquad\qquad
\cdot \left(\fint_{Q_{r} (y_0, s_0)} |f|^2 \right)^{1/2},
\endaligned
\end{equation}
where we have used (\ref{main-t}) for the last step.
By duality this implies that
\begin{equation}\label{w-1-23}
\aligned
&\left(\fint_{Q_{r}(y_0, s_0)} 
\big| e^{\psi (x)-\psi(y)}
\big( \Gamma_\e (x, t,; y, s) -\Gamma_0 (x, t; y, s) \big) \big|^2\, dyds\right)^{1/2}\\
&\le 
 C \e r^{-d-1} e^{c r \|\nabla \psi\|_\infty}
 \left\{ e^{cr^2 \|\nabla \psi\|_\infty^2}
+ e^{c|x_0-y_0|\|\nabla \psi\|_\infty}
\exp\left\{-\frac{\kappa |x_0-y_0|^2}{t_0-s_0}\right\}\right\}
\endaligned
\end{equation}
for any $(x, t)\in Q_{r}(x_0, t_0)$.

To deduce the $L^\infty$ bound for
$$
e^{\psi (x)-\psi (y)} \big(\Gamma_\e (x, t; y, s)-\Gamma_0 (x, t; y, s)\big)
$$
from its $L^2$ bound in (\ref{w-1-23}), we apply Theorem \ref{w-1-9} (with $\psi$ replaced by $-\psi$ and
$A$ replaced by $\widetilde{A}=\widetilde{A} (y, s)=A^*(y, -s)$) to the functions
$$
v_\e (y, s)=\Gamma_\e (x_0, t_0; y, -s) \quad
\text{ and } \quad v_0 (y,s)=\Gamma_0 (x_0, t_0, y, -s).
$$
Note that  $(\partial_t +\widetilde{\mathcal{L}}_\e)v_\e= (\partial_t +\widetilde{\mathcal{L}}_0)v_0=0$ in
$B(y_0, 3r)\times (-s_0-5r^2, -s_0+r^2)$.
Since $\widetilde{A}$ satisfies the same conditions as $A$, we obtain 
\begin{equation}\label{w-1-25}
\aligned
& |e^{\psi (x_0)-\psi (y_0)} (v_\e (y_0,- s_0)- v_0(y_0, -s_0))|\\
 & \le C e^{3r\|\nabla \psi\|_\infty} 
\left(\fint_{Q_r (y_0, -s_0)} |e^{\psi (x_0)-\psi (y)} (v_\e -v_0)|^2\, dy ds \right)^{1/2}\\
&\qquad + C \e r^{-d-1} e^{cr \|\nabla \psi\|_\infty} e^{\psi (x_0) -\psi (y_0)}
\exp\left\{-\frac{\kappa |x_0-y_0|^2}{t_0-s_0}\right\}\\
& = C e^{3r\|\nabla \psi\|_\infty} 
\left(\fint_{Q_r (y_0, s_0+r^2)} |e^{\psi (x_0)-\psi (y)} 
(\Gamma_\e (x_0, t_0; y, s)  -\Gamma_0 (x_0, t_0; y, s))|^2\, dy ds \right)^{1/2}\\
&\qquad + C \e r^{-d-1} e^{cr \|\nabla \psi\|_\infty} e^{\psi (x_0) -\psi (y_0)}
\exp\left\{-\frac{\kappa |x_0-y_0|^2}{t_0-s_0}\right\}\\
&\le 
 C \e r^{-d-1} e^{c r \|\nabla \psi\|_\infty}
 \left\{ e^{cr^2 \|\nabla \psi\|_\infty^2}
+ e^{c|x_0-y_0|\|\nabla \psi\|_\infty}
\exp\left\{-\frac{\kappa |x_0-y_0|^2}{t_0-s_0}\right\}\right\},
\endaligned
\end{equation}
where we have used (\ref{w-1-23}) for the last inequality.

Finally, as in \cite{Hofmann-2004, Dong-2008}, we let $\psi (y)= \gamma \psi_0 (|y-y_0|)$, where $\gamma\ge 0$ is to be chosen,
$\psi_0 (\rho)=\rho$ if $ \rho\le |x_0-y_0|$, and 
$\psi_0(\rho)=|x_0-y_0|$ if $\rho> |x_0-y_0|$.
Note that $\|\nabla \psi\|_\infty=\gamma$ and
$\psi (x_0)-\psi(y_0)=\gamma |x_0-y_0|$.
It follows from (\ref{w-1-25}) that
\begin{equation}\label{w-1-26}
\aligned
& |\Gamma_\e (x_0, t_0; y_0, s_0) -\Gamma_0 (x_0, t_0; y_0, s_0)|\\
&\le C \e r^{-d-1}
e^{-\gamma |x_0-y_0| +c \gamma  \sqrt{t_0-s_0}}
\left\{ e^{c \gamma^2 (t_0-s_0)}
+e^{c\gamma |x_0-y_0|} 
\exp\left\{-\frac{\kappa |x_0-y_0|^2}{t_0-s_0}\right\}\right\},
\endaligned
\end{equation}
where $c>0$ depends at most on $\mu$.
If $|x_0 -y_0|\le 2c  \sqrt{t_0-s_0}$, we may simply choose $\gamma=0$.
This gives
$$
\aligned
 |\Gamma_\e (x_0, t_0; y_0, s_0) -\Gamma_0 (x_0, t_0; y_0, s_0)
& \le C \e r^{-d-1}\\
&\le C \e (t_0-s_0)^{-\frac{d+1}{2}}\exp\left\{-\frac{\kappa |x_0-y_0|^2}{t_0-s_0}\right\}.
\endaligned
$$
If $|x_0-y_0|> 2c\sqrt{t_0-s_0}$, we choose
$$
\gamma=\frac{\delta |x_0-y_0|}{t_0-s_0}.
$$
Note that
$$
\aligned
 &-\gamma |x_0-y_0| +c\gamma \sqrt{t_0-s_0} +c\gamma^2 (t_0-s_0)\\
& =-\delta (1-c\delta) \frac{ |x_0-y_0|^2}{t_0-s_0} 
+ c\delta \frac{|x_0-y_0|}{\sqrt{t_0-s_0}}\\
&\le \Big\{ -\delta (1-c\delta) +(1/2)\delta  \Big\}
\frac{|x_0-y_0|^2}{t_0-s_0}\\
&\le \frac{-\delta |x_0-y_0|^2}{4(t_0-s_0)},
\endaligned
$$
if $ \delta \le (1/4)c^{-1}$. Also, observe that
$$
\aligned
 c\gamma \sqrt{t_0-s_0} +c\gamma |x_0 -y_0| -\frac{\kappa |x_0-y_0|^2}{t_0-s_0}
&\le \Big\{ (1/2) \delta +c\delta -\kappa \Big\} \frac{|x_0-y_0|^2}{t_0-s_0}\\
& \le -\frac{\kappa |x_0-y_0|^2}{2(t_0-s_0)},
\endaligned
$$
if $\delta\le (1/2)(c+1/2)^{-1}\kappa$. Recall that $r=(100)^{-1}\sqrt{t_0-s_0}$.
As a result, we have proved that there exists $\kappa_1>0$, depending only on $\mu$, such that
$$
|\Gamma_\e (x_0, t_0; y_0, s_0)-\Gamma_0 (x_0, t_0; y_0, s_0)|
\le \frac{C \e}{(t_0-s_0)^{\frac{d+1}{2}}}
\exp \left\{ -\frac{\kappa_1 |x_0-y_0|^2}{t_0-s_0}\right\}.
$$
This completes the proof of Theorem \ref{main-theorem-1}.
\end{proof}



\section{Proof of Theorem \ref{main-theorem-2}}\label{section-6}

Define
$$
\| F\|_{C^{\lambda, 0} (K)} =\sup \left\{ \frac{|F(x, t)-F(y, t)|}{|x-y|^\lambda}: \
(x, t), (y, t)\in K \text{ and } x\neq y \right\}.
$$
The proof of Theorem \ref{main-theorem-2} relies on the following Lipschitz estimate.

\begin{thm}\label{w-Lip}
Assume  that $A$ satisfies conditions (\ref{ellipticity}), (\ref{periodicity}) and (\ref{H}).
Suppose that 
$$
(\partial_t +\mathcal{L}_\e)u_\e =(\partial_t +\mathcal{L}_0)u_0
$$
in $Q_{2r} (x_0, t_0)$ for some $(x_0, t_0)\in \R^{d+1}$ and $\e \le r <\infty$.
Then
\begin{equation}\label{w-Lip-1}
\aligned
& \| \nabla u_\e -\nabla u_0 -(\nabla\chi)^\e \nabla u_0\|_{L^\infty (Q_r(x_0, t_0))}\\
&\le C  r^{-1} 
\left(\fint_{Q_{2r} (x_0, t_0)} |u_\e -u_0|^2 \right)^{1/2}
+C \e r^{-1} \|\nabla u_0\|_{L^\infty(Q_{2r} (x_0, t_0))}\\
& \qquad
+C \e \ln  \big[ \e^{-1} r +2\big]  
\|  |\nabla^2 u_0|
+\e |\partial_t \nabla u_0|
+\e |\nabla^3 u_0| \|_{L^\infty (Q_{2r} (x_0, t_0))} \\
&\qquad +  C \e^{1+\lambda}
\|  |\nabla^2 u_0|
+ \e |\partial_ t \nabla u_0|
+\e  |\nabla^3 u_0| \|_{C^{\lambda, 0} (Q_{2r} (x_0, t_0))},
\endaligned
\end{equation}
where $C$ depends only on $d$, $m$, $\mu$ and $(\lambda, \tau)$ in (\ref{H}).
\end{thm}

\begin{proof}
Let
\begin{equation}\label{w-Lip-2}
w_\e =u_\e -u_0 -\e \chi_j^\e  \frac{\partial u_0}{\partial x_j} -\e^2 \phi^\e_{(d+1)ij} \frac{\partial^2 u_0}{\partial x_i\partial x_j},
\end{equation}
where $\chi_j^\e (x, t)=\chi_j (x/\e, t/\e^2)$ and $\phi_{(d+1)ij}^\e (x, t) = \phi_{(d+1)ij}(x/\e, t/\e^2)$.
It follows by Proposition 3.1 that
$ (\partial_t +\mathcal{L}_\e) w_\e = \e\, \text{\rm div} (F_\e)$ in $Q_{2r} (x_0, t_0)$,
where $F_\e$ is given by (\ref{e-4}) with $S_\e$ being the identity operator.
Choose a cut-off function $\varphi \in C(\R^{d+1})$ such that
$$
\left\{
\aligned
& 0\le \varphi\le 1, \ \ \varphi =1 \text{ in } Q_{3r/2} (x_0, t_0),\\
& \varphi (x, t)=0 \text{  if } |x-x_0|\ge (7r/4) \text{ or } t<t_0-(7r/4)^2, \\
& |\nabla \varphi|\le Cr^{-1}, \ |\nabla^2 \varphi| +|\partial_t \varphi|\le C r^{-2}.
\endaligned
\right.
$$
Using
$$
\aligned
(\partial_t +\mathcal{L}_\e)(\varphi w_\e)
&=(\partial_t \varphi) w_\e +\e\, \text{\rm div}(\varphi F_\e)
-\e F_\e (\nabla \varphi)\\
& \qquad-\text{\rm div} (A^\e (\nabla \varphi) w_\e)
-A^\e \nabla w_\e  (\nabla \varphi),
\endaligned
$$
where $A^\e (x, t)=A(x/\e, t/\e^2)$,
we may deduce that for any $(x, t)\in Q_r (x_0, t_0)$,
$$
\aligned
 w_\e (x, t)
&= \int_{-\infty}^t \int_{\R^d}  \Gamma_\e (x, t; y, s)
\big\{ (\partial_s \varphi ) w_\e
-\e F_\e (\nabla \varphi)
- A^\e \nabla w_\e (\nabla \varphi)\big\}\, dyds\\
&\qquad-
\int_{-\infty}^t \int_{\R^d} \nabla_y \Gamma_\e (x, t; y, s) 
\big\{ \e \varphi F_\e 
+ A^\e (\nabla \varphi) w_\e \big\}\, dyds\\
&=I(x, t) +J(x, t),
\endaligned
$$
where
$$
J(x, t)=-\e \int_{-\infty}^t \int_{\R^d} \nabla_y \Gamma_\e (x, t; y, s) \varphi (y, s) F_\e (y, s)\, dyds.
$$
Since $\varphi=1$ in $Q_{3r/2}(x_0, t_0)$,
we see that  for $(x, t)\in Q_r(x_0, t_0)$,
$$
\aligned
|\nabla I(x, t)|
&\le C \int_{-\infty}^t \int_{\R^d}
|\nabla_x \Gamma_\e (x, t; y, s)|
\big\{ |\partial_s \varphi| |w_\e|
+\e |F_\e| |\nabla \varphi|
+ |\nabla w_\e | |\nabla \varphi|\big\} dyds\\
&\qquad +C \int_{-\infty}^ t \int_{R^d} |\nabla_x\nabla_y \Gamma_\e (x, t; y, s)|
 |\nabla \varphi| |w_\e|\, dyds\\
&\le 
C\left\{
\frac{1}{r} \fint_{Q_{2r} (x_0, t_0)} |w_\e|
+\e \fint_{Q_{2r}(x_0, t_0)} |F_\e|
+ \fint_{Q_{7r/4}(x_0, t_0)} |\nabla w_\e| \right\}\\
&\le 
C\left\{
\frac{1}{r} \left(\fint_{Q_{2r} (x_0, t_0)} |w_\e|^2\right)^{1/2}
+\e \left(\fint_{Q_{2r}(x_0, t_0)} |F_\e|^2\right)^{1/2}\right\},
\endaligned
$$
where we have used (parabolic) Cacciopoli's inequality  for the last step.
In view of (\ref{e-4}) with $S_\e$ being the identity operator,
$$
|F_\e|\le C \big\{ |\nabla^2 u_0| +\e |\partial_t \nabla u_0| + \e |\nabla^3 u_0| \big\},
$$
where we have used the boundedness of $\nabla \phi$ (see Remark \ref{re-2.0}).
It follows that $\|\nabla I\|_{L^\infty (Q_r(x_0, t_0))}$ is bounded by the RHS of (\ref{w-Lip-1}).

Finally, to estimate $J(x, t)$, we write
$$
\aligned
J(x, t)
= &-\e \int_{-\infty}^t \int_{\R^d} \nabla_y \big\{ \Gamma_\e (x, t; y, s)\varphi (y, s)\big\}
\big(F_\e (y, s)-F_\e (x, s) \big)\, dyds\\
&\qquad
+\e \int_{-\infty}^t \int_{\R^d} \Gamma_\e (x, t; y, s) (\nabla \varphi) (y, s) F_\e (y, s)\, dyds.
\endaligned
$$
It follows that for $(x, t)\in Q_{r}(x_0, t_0)$,
\begin{equation}\label{6-10}
\aligned
|\nabla J(x, t)|
&\le \e \int_{Q_{2r} (x_0, t_0)} 
|\nabla_x\nabla_y \big\{\Gamma_\e (x, t; y, s)\varphi (y, s)\big\} |
|F_\e (y, s)-F_\e (x, s) |\, dyds\\
&\qquad + \e \int_{Q_{2r}(x_0, t_0)} 
|\nabla_x \Gamma_\e (x, t; y, s)| |\nabla \varphi (y, s)| |F_\e (y,s)|\, dyds\\
&\le  C \e \int_{Q_{2r} (x_0, t_0)}
\frac{|F_\e (y, s)-F_\e (x, s)|}{\left(|x-y| +|t-s|^{1/2}\right)^{d+2}}\, dyds
+ C \e \fint_{Q_{2r}(x_0, t_0)} |F_\e|
\endaligned
\end{equation}
To bound the first integral in the RHS of (\ref{6-10}), we subdivide   the domain
$Q_{2r} (x_0, t_0)$ into $Q_\e (x, t)$ and $Q_{2r} (x_0, t_0)\setminus Q_\e (x, t)$.
On $Q_{2r} (x_0, t_0)\setminus Q_\e (x, t)$, we use the bound
$$
|F_\e (y, s)-F_\e (x, s)|\le 2\| F_\e\|_{L^\infty(Q_{2r}(x_0, t_0))},
$$
while for $Q_\e (x, t)$, we use
$$
|F_\e (y, s)-F_\e (x, s)|\le |x-y|^\lambda  \| F\|_{C^{\lambda, 0} (Q_{2r}(x_0, t_0))}.
$$
This leads to 
$$
\aligned
|\nabla J(x, t)|
&\le C \e \ln \big[ \e^{-1} r +1\big] \|F_\e\|_{L^\infty(Q_{2r}(x_0, t_0))}
+ C \e^{1+\lambda} \|F_\e \|_{C^{\lambda, 0}(Q_{2r} (x_0, t_0))}\\
&\le C \e \ln  \big[ \e^{-1} r +1\big]  
\|  |\nabla^2 u_0|
+\e |\partial_t \nabla u_0|
+\e |\nabla^3 u_0| \|_{L^\infty (Q_{2r} (x_0, t_0))} \\
&\quad +  C \e^{1+\lambda}
\|  |\nabla^2 u_0|
+ \e |\partial_ t \nabla u_0|
+\e  |\nabla^3 u_0| \|_{C^{\lambda, 0} (Q_{2r} (x_0, t_0))}.
\endaligned
$$
Thus, in view of the estimate for $\nabla I (x, t)$, we have proved
that $\|\nabla w_\e\|_{L^\infty(Q_r(x_0, t_0))}$ is bounded by the RHS of (\ref{w-Lip-1}).
Since
$$
\|\nabla w_\e - \big\{ \nabla u_\e -\nabla u_0 -(\nabla \chi)^\e \nabla u_0\big\} \|_{L^\infty(Q_r(x_0, t_0))}
\le C \e \| |\nabla^2 u_0| +\e |\nabla^3 u_0| \|_{L^\infty(Q_r(x_0, t_0))},
$$
the estimate (\ref{w-Lip-1}) follows.
\end{proof}

To prove Theorem \ref{main-theorem-2},
we fix $x_0, y_0\in \R^d$ and $s_0<t_0$.
We may assume that $ \e < (t_0-s_0)/8$.
For otherwise the estimate (\ref{f-0-4}) follows directly from (\ref{D-1}).
We apply Theorem \ref{w-Lip} to the functions
$$
u_\e (x, t)=\Gamma_\e (x, t; y_0, s_0) \quad
\text{ and } \quad
u_0 (x, t)= \Gamma_0 (x, t; y_0, s_0)
$$
in $Q_{2r} (x_0, t_0)$, where $r= (t_0 -s_0)/8$.
Note that $(\partial_t +\mathcal{L}_\e)u_\e =(\partial_t +\mathcal{L}_0) u_0=0$
in $Q_{4r}(x_0, t_0)$.
To bound the first term in the RHS of (\ref{w-Lip-1}), we use the estimate (\ref{f-0-3})
in Theorem \ref{main-theorem-1}.
All other terms in the RHS of (\ref{w-Lip-1}) may be handled easily by 
using the estimates(\ref{f-const}) for $\Gamma_0(x, , t; y, s)$.
We leave the details to the reader.



\section{Proof of Theorem \ref{main-theorem-3}}\label{section-7}

To prove Theorem \ref{main-theorem-3},
we fix $x_0, y_0\in \R^d$ and $s_0<t_0$.
As before, we may assume that $\e<(t_0-s_0)/8$.
For otherwise the estimate (\ref{f-0-9}) follows directly from (\ref{D-2}).

Let $r=(t_0-s_0)/8$. Fix $1\le j\le d$ and $1\le \beta\le m$.
We apply Theorem \ref{w-Lip} to the functions $u_\e =(u_\e^\alpha)$ and
$u_0=(u_0^\alpha)$ in $Q_{2r} (x_0, t_0)$, where
$$
\aligned
u_\e^\alpha (x, t)  & =\frac{\partial }{\partial y_j} \Big\{ \Gamma_\e^{\alpha\beta} \Big\} (x, t; y_0, s_0)\\
u_0^\alpha (x, t)& =\frac{\partial}{\partial y_\ell} \Big\{ \Gamma_0^{\alpha\sigma} \Big\}
(x, t; y_0, s_0) \cdot \left\{ \delta^{\beta \sigma} \delta_{j\ell} 
+\frac{\partial }{\partial y_j} (\widetilde{\chi}_\ell^{\beta \sigma}) (y_0/\e, -s_0/\e^2) \right\},
\endaligned
$$
where $\widetilde{\chi}$ denotes the correctors for $\partial_t  +\widetilde{\mathcal{L}}_\e$.
Observe that $(\partial_t +\mathcal{L}_\e)u_\e =(\partial_t +\mathcal{L}_0) u_0=0$
in $Q_{4r}(x_0, t_0)$.
To bound the first term in the RHS of (\ref{w-Lip-1}),
we use the estimate (\ref{f-0-8}).
As in the proof of Theorem \ref{main-theorem-1},
all other terms in the RHS of (\ref{w-Lip-1}) may be handled readily by using estimate (\ref{f-const})
for $\Gamma_0(x, t; y, s)$.

 \bibliographystyle{amsplain}
 
\bibliography{Geng-Shen-2017.bbl}

\bigskip

\begin{flushleft}
Jun Geng,
School of Mathematics and Statistics,
Lanzhou University,
Lanzhou, P.R. China.

E-mail:gengjun@lzu.edu.cn
\end{flushleft}

\begin{flushleft}
Zhongwei Shen,
Department of Mathematics,
University of Kentucky,
Lexington, Kentucky 40506,
USA.

E-mail: zshen2@uky.edu
\end{flushleft}

\bigskip

\medskip

\end{document}